\documentclass[reqno]{amsart}
\usepackage{amscd}
\usepackage{amssymb}
\usepackage[dvips]{graphics}
\usepackage[margin=1in,includeheadfoot]{geometry}
\usepackage{url}
\usepackage{hyperref}
\usepackage{cite}

\def\beq{\begin{equation}}
\def\eeq{\end{equation}}
\def\ba{\begin{array}}
\def\ea{\end{array}}

\newtheorem{thm}{Theorem}[section]
\newtheorem{lm}[thm]{Lemma}

\newtheorem{crl}[thm]{Corollary}

\theoremstyle{definition}
\newtheorem{rem}[thm]{Remark}

\newtheorem{df}[thm]{Definition}
\newtheorem{ex}[thm]{Example}
\theoremstyle{remark}

\begin{document}
\pagestyle{plain}
\title{$C^{0}$-regularity for solutions of elliptic equations with distributional coefficients}

\author{Jingqi Liang\\
 \small{School of Mathematical Sciences, Shanghai Jiao Tong University}\\
 \small{Shanghai, China}\\\\
\small{ Lihe Wang}\\
 \small{Department of Mathematics, University of Iowa, Iowa City, IA, USA;\\ School of Mathematical Sciences, Shanghai Jiao Tong University}\\
 \small{Shanghai, China}\\\\
\small{Chunqin Zhou}\\
 \small{School of Mathematical Sciences, MOE-LSC, CMA-Shanghai, Shanghai Jiao Tong University}\\
 \small{Shanghai, China}}

\footnote{The third author was partially supported by NSFC Grant 12031012.}

\begin{abstract}
In this paper, the continuity of solutions for elliptic equations in divergence form with distributional coefficients is considered. Inspired by the discussion on necessary and sufficient conditions for the form boundedness of elliptic operators by Maz'ya and Verbitsky (Acta Math. \textbf{188}, 263-302, 2002 and Comm. Pure Appl. Math. \textbf{59}, 1286-1329, 2006), we propose two kinds of sufficient conditions, which are some Dini decay conditions and some integrable conditions named Kato class or $K^{1}$ class,  to  show that the weak solution of the Schr\"{o}dinger type elliptic equation with distributional coefficients is continuous, and give an almost optimal priori estimate. These estimates can clearly show that how the coefficients and nonhomogeneous terms influence the regularity of solutions. The $\ln$-Lipschitz regularity and H\"{o}lder regularity are also obtained as corollaries which cover the classical De Giorgi's H\"{o}lder estimates.
\end{abstract}

\maketitle


\section{Introduction}
The goal of the present paper is to obtain the pointwise continuity for $W^{1,2}$ weak solutions of the following divergence form equation with distributional coefficients,
\begin{equation}\label{1}
-\Delta u+Vu=-\text{div}\vec{f}+g\qquad \text{in}~\Omega,
\end{equation}
where $\Omega$ is an open bounded subset of $\mathbb{R}^{n}$ with $n\geq3$.

In the past two decades, the differential operator with distributional coefficients attract many scholars' attention, such as the following second order differential operator with distributional coefficients acting from $W^{1,2}(\mathbb{R}^{n})$ to $W^{-1,2}(\mathbb{R}^{n})$:
\begin{equation}\label{gsl}
\mathcal{L}=\text{div}(A\nabla )+\vec{b}\cdot\nabla +q,
\end{equation}
where $A=(a_{ij})_{n\times n}$, $\vec{b}=(b_{1},b_{2},\cdots,b_{n})$, $a_{ij},~b_{i},~q$ are real or complex valued distributions on $W^{1,2}(\mathbb{R}^{n})$. The operator $\mathcal{L}:~W^{1,2}(\mathbb{R}^{n})\rightarrow W^{-1,2}(\mathbb{R}^{n})$ is said to be bounded if and only if the sesquilinear inequality
\begin{equation}\label{gfbd}
|\langle \mathcal{L}u,v\rangle|\leq C\|u\|_{W^{1,2}(\mathbb{R}^{n})}\|v\|_{W^{1,2}(\mathbb{R}^{n})}
\end{equation}
holds for all $u,v\in C_{0}^{\infty}(\mathbb{R}^{n})$, where $C$ does not depend on $u,v\in C_{0}^{\infty}(\mathbb{R}^{n})$.

Maz'ya and Verbitsky \cite{MV02,MV06} have already characterized $a_{ij},~b_{i},~c$ to get the necessary and sufficient condition on the form boundedness of operator $\mathcal{L}$ where harmonic analysis and potential theory methods were employed. In 2002, they firstly considered the Schr\"{o}dinger operator, $\mathcal{L}=-\text{div}(\nabla)+V$. In this situation, the form boundedness of $\mathcal{L}$ is equivalent to the boundedness of $V$. They proved that the sesquilinear form defined by $\langle Vu,v\rangle=\langle V,uv\rangle$ is bounded on $W^{1,2}(\mathbb{R}^{n})\times W^{1,2}(\mathbb{R}^{n})$
if and only if there exists a vector field $\vec{\Gamma}=(\Gamma_{1},\Gamma_{2},\cdots,\Gamma_{n})\in L_{\text{loc}}^{2}(\mathbb{R}^{n})^{n}$ and $\Gamma_{0}\in L_{\text{loc}}^{2}(\mathbb{R}^{n})$ such that
$$V=\text{div}\vec{\Gamma}+\Gamma_{0}
$$
and $|\Gamma_{i}|^{2}(i=0,1,\cdots,n)$ are admissible measure for $W^{1,2}(\mathbb{R}^{n})$, i.e.
$$\int_{\mathbb{R}^{n}}|u(x)|^{2}|\Gamma_{i}(x)|^{2}dx\leq C\|u\|_{W^{1,2}(\mathbb{R}^{n})}^{2},\quad i=0,1,\cdots,n,
$$
where $C$ does not depend on $u\in C_{0}^{\infty}(\mathbb{R}^{n})$. In 2006 , they generalized their results to the general second order differential operator $\mathcal{L}$ in \cite{MV06} and in this case there were no ellipticity assumptions on $A$.

Actually, if the the sesquilinear mapping associated with $\mathcal{L}$ is bounded, then the $W^{1,2}$ weak solutions can be defined for some equations that $\mathcal{L}u$ satisfies. In other words, the discussion on the necessary and sufficient conditions on the form boundedness of differential operators can guarantee the existence of weak solutions for some equations with such distributional coefficients.

Actually, in  2012, Jaye, Maz'ya and Verbitsky studied the homogeneous equation of Schr$\ddot{\text{o}}$dinger type
\begin{equation}\label{2012}
-\text{div}(A\nabla u)-\sigma u=0\quad \text{in}~\Omega,
\end{equation}
where $\Omega$ is a domain in $\mathbb{R}^{n}$, $A\in L^{\infty}(\Omega)^{n\times n}$ satisfies ellipticity assumption, $\sigma\in D'(\Omega)$ is a real-valued distributional potential. They showed the existence and the optimal regularity of positive solutions for above equation: if there are upper and lower bounds of $\langle\sigma,h^{2}\rangle$ with $\lambda<1$ and $\Lambda>0$, that is,
$$\langle\sigma,h^{2}\rangle\leq \lambda\int_{\Omega}(A\nabla h)\cdot\nabla hdx,\quad\langle\sigma,h^{2}\rangle\geq -\Lambda\int_{\Omega}(A\nabla h)\cdot\nabla hdx,\quad \text{for all}~h\in C_{0}^{\infty}(\Omega),
$$
then there exists a positive $W^{1,2}_{\text{loc}}$ solution of $(\ref{2012})$ and the $W^{1,2}_{\text{loc}}$ regularity is optimal. Similarly they extend their results to a quasilinear version for operators of the $p$-Laplace type in \cite{JMV2013} and also got the existence and optimal local Sobolev regularity of positive solutions under a mild restriction on $\sigma$. Besides, analogous problems have been studied in fractional Sobolev spaces, infinitesimal form boundedness, which can be found in \cite{MV021,MV04,MV05}.

To our knowledge, we observe that the scholars focus more on the existence of solutions, merely on the regularity of solutions for the equations with distributional coefficients even for the homogeneous equation. It is natural to ask a question: under what conditions on $V,~\vec{f},~g$ the solution of $(\ref{1})$ will be continuous or H\"{o}lder continuous, even higher regularity?
\
\

The present paper is devoted to obtain the continuity of solutions. To this purpose, we firstly recall some known results when the coefficients $V$ and the nonhomogeneous terms $\vec{f},~g$ are measurable functions and locally integrable.

For classical Schr\"{o}dinger operator $L$ of the form
\begin{equation}\label{generall}
Lu=-\text{div}(\nabla u)+Vu,
\end{equation}
where $V$ is measurable on a bounded domain $\Omega\subset\mathbb{R}^{n}(n\geq3)$, it is well-known that if $V\in L^{\frac{n}{2}}(\Omega)$, then for any $u\in W^{1,2}(\Omega)$, $v\in W_{0}^{1,2}(\Omega)$, by denoting
$$a(u,v)=\int_{\Omega}\nabla u\cdot\nabla v+Vuvdx,
$$
it is easy to verify $a(\cdot,\cdot)$ is a bounded bilinear mapping on $W^{1,2}(\Omega)\times W_{0}^{1,2}(\Omega)$. It means that $V\in L^{\frac{n}{2}}(\Omega)$ is a sufficient condition to guarantee the boundedness of $a(\cdot,\cdot)$. Then the weak solution of the inhomogeneous equation
\begin{equation}\label{genel}
Lu=-\text{div}\vec{f}+g\qquad \text{in}~\Omega,
\end{equation}
can be defined reasonably provided that $\vec{f}\in L^{2}(\Omega)^{n}$, $g\in L^{\frac{2n}{n+2}}(\Omega)$. Furthermore, if $V\in L^{\frac{q}{2}}(\Omega)$ with $q>n$ and $f_{i}\in L^{q}(\Omega)$, $g\in L^{\frac{nq}{n+q}}(\Omega)$, then local maximum principle, Harnack inequality, interior H\"{o}lder regularity and existence theory are well-known, see \cite{DT}.
The similar results also hold for the second order uniformly elliptic equation:
\begin{equation}\label{generall}
-D_{j}(a_{ij}D_{i}u)+b_{i}D_{i}u+cu=-\text{div}\vec{f}+g.
\end{equation}

Besides, in 1980s, many scholars try to generalize the classical H\"{o}lder regularity or the continuity of solution under weaker integrable assumptions on $V$: such as that $V$ belongs to Kato class (see Definition $\ref{ka}$), some Morrey spaces, some Lorentz spaces and so on. Kato class was firstly introduced by Aizenman and Simon in \cite{AS1982}, which is based on a condition considered by Kato in \cite{Kato1973}. They use the probabilistic technique to show Harnack inequality and the continuity of solutions for $-\Delta u+Vu=0$ while $V$ belongs to the Kato class, denoted by $V\in K(\Omega)$ in the following. Later, Simader \cite{Si1990}, Hinz and Kalf \cite{HK1990} proved same result by a different method: they use the Green function of $\Delta$ to represent the solutions locally. Instead of Laplacian, Chiarenza, Fabes and Garofalo \cite{CFG1986} considered a general uniformly elliptic operator in divergence form $Lu=D_{j}(a_{ij}(x)D_{i}u)+Vu$, $V\in K(\Omega)$, they proved the continuity and uniform Harnack inequality for solutions based on real variable approach which deeply  depends on $L^{p}$ estimates of Green functions given by Fabes and Stroock in \cite{FS1984}.  Kurata \cite{K1994} developed the method of Chiarenza to prove the local boundedness, Harnack inequality and continuity for weak solutions of general second order elliptic equations with bounded measurable coefficients: $-\text{div}(A(x)\nabla u)+\vec{b}\cdot \nabla u+Vu=0$, the main assumptions are that $V$ and $|\vec{b}|^{2}$ belong to the local Kato class. For other general uniformly elliptic and degenerate elliptic operators with lower term coefficients satisfying Kato type conditions, the solution can also obey local boundedness principle, be continuous and satisfy Harnack inequality, we refer the readers to \cite{CGL1993,Gu1989,Mo1998,Mo2000,Mo2002}. Except Kato type conditions, Di Fazio \cite{D1988} study the same equation as Aizenman's but assume $V$ in Morrey space $L^{1,\mu}$ where $\mu>\displaystyle\frac{n-2}{n}$, they proved H\"{o}lder inequality and improved the continuity result in \cite{CFG1986}. For more properties on Kato class and the relationship between Kato class and Morrey spaces, Lorentz spaces, we refer the readers to \cite{CRR2017,CRR2019,DH1998,FGL1990,THG2020}.

In this paper, we will propose two kinds of sufficient conditions for $V,~\vec{f},~g$ to show the $C^0$-regularity of weak solutions of (\ref{1}). Particularly, to show $C^0$-regularity of weak solutions, we assume firstly some Dini decay conditions for $V,~\vec{f},~g$ in each ball $B_r(x_0)$, and then we assume that $V,~g$ belong to Kato class and $\vec{f}$ belongs to $K^{1}$ class (see Definition $\ref{ka}$). Our main results are following.


\begin{thm}\label{linfty}Let $\Omega$ be a bounded domain and $0\in\Omega$. Assume that $V\in M(W^{1,2}(\Omega)\rightarrow W^{-1,2}(\Omega))$, i.e. $\langle V\cdot,\cdot\rangle: W^{1,2}(\Omega)\times W_{0}^{1,2}(\Omega)\rightarrow \mathbb{R}$ is a bounded bilinear mapping, $\vec{f}\in L^{2}(\Omega)^{n}$ satisfying that $|\vec{f}|^2$ is an admissible measure for $W_{0}^{1,2}(\Omega)$, and $g\in W^{-1,2}(\Omega)$. Suppose that  $u\in W^{1,2}(\Omega)$ is a weak solution of ($\ref{1}$) in $\Omega$.
If there exists a positive constant $R\leq1$ with $B_{R}\subset\Omega$
such that for any $0<r\leq \displaystyle\frac{R}{2}$ and $\psi\in W^{1,2}(\Omega)$, $\varphi\in W_{0}^{1,2}(\Omega)$ with $\text{supp}\{\varphi\}\subset \overline{B}_{r}$,
\begin{eqnarray*}
& & |\langle V\psi,\varphi\rangle|\leq \omega_{1}(s)\|\psi\|_{L_{r,s}^{1,2}}\|\nabla\varphi\|_{L^{2}(B_{r})}, ~~~ \forall r<s\leq 2r,\\
& & \displaystyle\int_{B_{r}}|\vec{f}|^{2}|\varphi|^{2}dx\leq (\omega_{2}(r))^{2}\displaystyle\int_{B_{r}}|\nabla\varphi|^{2}dx,\\
& & |\langle g,\varphi\rangle|\leq \frac{\omega_{2}(r)}{r^{2}}|B_{r}|^{\frac{n+2}{2n}}\|\nabla\varphi\|_{L^{2}(B_{r})},
\end{eqnarray*}
where $\|\psi\|_{L_{r,s}^{1,2}}=\displaystyle\frac{\|\psi\|_{L^{2}(B_{s})}}{s-r}+\|\nabla \psi\|_{L^{2}(B_{s})}$,  $\omega_{i}(r)$ is Dini modulus of continuity satisfying $\displaystyle\int_{0}^{R} \frac{\omega_{i}(r)}{r} d r<\infty$ for $i=1,2$, then $u$ is continuous at 0 in the $L^{2}$ sense. Moreover, there is a constant $K$ such that for any $0<r\leq R$,
\begin{eqnarray*}
& & \left(\frac{1}{|B_{r}|}\displaystyle\int_{B_{r}}|u-K|^{2}dx\right)^{\frac{1}{2}}\leq
\Omega_{1}(r)+\Omega_{2}(r),\end{eqnarray*}
and
$$ \left|K\right|\leq CA,$$
where
\begin{eqnarray*}
\Omega_{1}(r)=\begin{cases}
CA
\left(\displaystyle\frac{r}{\bar{r}}+r
\displaystyle\int_{r}^{\bar{r}}\frac{\omega_{1}(s)}{s^{2}}ds
+\displaystyle\int_{0}^{r}\frac{\omega_{1}(s)}{s}ds\right), \quad &0<r\leq\bar{r},\\
CA,\quad &\bar{r}<r\leq R,\\
\end{cases}
\end{eqnarray*}
\begin{eqnarray*}
\Omega_{2}(r)=\begin{cases}
C
\left(r
\displaystyle\int_{r}^{\bar{r}}\frac{\omega_{2}(s)}{s^{2}}ds
+\displaystyle\int_{0}^{r}\frac{\omega_{2}(s)}{s}ds\right), \quad &0<r\leq\bar{r},\\
CA,\quad &\bar{r}<r\leq R,\\
\end{cases}
\end{eqnarray*}

\begin{eqnarray*}
A=\left(\displaystyle\frac{1}{|B_{\bar{r}}|}\int_{B_{R}}u^{2}dx
\right)^{\frac{1}{2}}+\frac{4|B_{1}|^{\frac{n+2}{2n}}}{\delta_{0}}\left(\omega_{2}(R)
+\frac{1}{1-\lambda}\displaystyle\int_{0}^{R}\frac{\omega_{2}(s)}{s}ds\right),
\end{eqnarray*}
and $\bar{r}$ is chosen to be the solution of
\begin{equation*}
C_{0}|B_{1}|^{\frac{n+2}{2n}}
\left(\omega_{1}(\bar{r})+\frac{1}{1-\lambda}\int_{0}^{\bar{r}}\frac{\omega_{1}(s)}{s}ds\right)=
\frac{\delta_{0}}{128}.
\end{equation*}
Here $C$ is a universal constant $C=C(n)$,  and $\lambda,\delta_{0},C_{0}$ are the constants in Lemma $\ref{key1}$.
\end{thm}

\begin{thm}\label{thm1.2}
For any $g\in K_{\eta_{g}}(B_{1})$, $\vec{f}\in K^{1}_{\eta_{f}}(B_{1})\cap L^{2}(B_{1})^{n}$, $V\in K_{\eta_{V}}(B_{1})$ with
$$\|V\|_{K(B_{1})}\leq\delta
$$
for some $\delta$ sufficiently small, if $u\in W^{1,2}(B_{1})$ is a weak solution of $(\ref{1})$ in $B_{1}$,
then $u\in C(B_{1})$ and $u$ is locally bounded with the estimate:
\begin{eqnarray*}
\|u\|_{L^{\infty}(B_{\frac{1}{2}})}\leq C\left(\|u\|_{L^{2}(B_{1})}+\|\vec{f}\|_{L^{2}(B_{1})}+\|\vec{f}\|_{K^{1}(B_{1})}+\|g\|_{K(B_{1})}\right).
\end{eqnarray*}
\end{thm}

With respect to the modulus of continuity, we refer the readers to \cite{K1997,K,2} for more details. Here we give some main properties as a complement.
\begin{rem}\label{remark1.3}
Any modulus of continuity $\omega(t)$ is non-decreasing, subadditive, continuous and satisfies $\omega(0)=0$. Hence any modulus of continuity $\omega(t)$ satisfies
\begin{eqnarray*}
\frac{\omega(r)}{r}\leq2\frac{\omega(h)}{h},\quad 0<h<r.
\end{eqnarray*}
\end{rem}

Theorem $\ref{linfty}$ is a pointwise regularity result. It not only implies the continuity of the solution, but also shows a priori estimate which clearly illustrates how the distributional coefficients and the nonhomogeneous terms influence the behavior of the solution near zero point.  This theorem also holds for other points inside $\Omega$, and the classical interior continuity can be proved straightforward by this pointwise continuity. This theorem provides a frame theory, especially when the modulus of continuity is H$\ddot{\text{o}}$lder continuous, above estimates imply that the solution is H$\ddot{\text{o}}$lder continuous, which cover the De Giorgi's H$\ddot{\text{o}}$lder regularity. Let us give some corollaries and remarks as follows.

\begin{crl}\label{conti}
We assume $\Omega'\Subset\Omega$. Under the assumptions of Theorem $\ref{linfty}$, furthermore, there exists a positive constant $R\leq\min\{1,\text{dist}(\Omega',\Omega)\}$ with $B_{R}(x_{0})\subset\Omega$ for any $x_{0}\in\Omega'$ such that for any $0<r\leq \displaystyle\frac{R}{2}$ and $\psi\in W^{1,2}(\Omega)$, $\varphi\in W_{0}^{1,2}(\Omega)$ with $\text{supp}\{\varphi\}\subset \overline{B_{r}(x_{0})}$,
$$
|\langle V\psi,\varphi\rangle|\leq \omega_{1}(s)\|\psi(\cdot+x_{0})\|_{L_{r,s}^{1,2}}\|\nabla\varphi\|_{L^{2}(B_{r}(x_{0}))},~~~ \forall r<s\leq 2r
$$
$$
\displaystyle\int_{B_{r}(x_{0})}|\vec{f}|^{2}|\varphi|^{2}dx\leq (\omega_{2}(r))^{2}\displaystyle\int_{B_{r}(x_{0})}|\nabla\varphi|^{2}dx,
$$
$$
|\langle g,\varphi\rangle|\leq \frac{\omega_{2}(r)}{r^{2}}|B_{r}|^{\frac{n+2}{2n}}\|\nabla\varphi\|_{L^{2}(B_{r})},
$$
where $\omega_{i}(r)$ is Dini modulus of continuity satisfying $\displaystyle\int_{0}^{R} \frac{\omega_{i}(r)}{r} d r<\infty$ for $i=1,2$.
If $u\in W^{1,2}(\Omega)$ is a weak solution of ($\ref{1}$), then $u$ is classical continuous in $\Omega'$ after necessary modification on a measure zero set. It follows that $u\in L^{\infty}(\Omega')$ and
$$\|u\|_{L^{\infty}(\Omega')}\leq C\left(\frac{\|u\|_{L^{2}(\Omega)}}{|B_{\bar{r}}|^{\frac{1}{2}}}+\omega_{2}(R)+\int_{0}^{R}\frac{\omega_{2}(r)}{r}dr\right),
$$
where $C$ depends on $n,\omega_{1},\Omega$, $R$, and $\bar{r}$ as defined in previous theorem.
\end{crl}
\begin{crl}\label{holder}
Under the assumptions of Theorem $\ref{linfty}$, moreover there exists a positive constant $R\leq1$ with $B_{R}\subset\Omega$, $0<\alpha_{i}<1$, $N_{i}>0$, $i=1,2$ such that
for any $0<r\leq \displaystyle\frac{R}{2}$ and $\psi\in W^{1,2}(\Omega)$, $\varphi\in W_{0}^{1,2}(\Omega)$ with $\text{supp}\{\varphi\}\subset \overline{B_{r}}$,
$$
|\langle V\psi,\varphi\rangle|\leq N_{1}r^{\alpha_{1}}\|\psi\|_{L_{r,s}^{1,2}}\|\nabla\varphi\|_{L^{2}(B_{r})}, ~~~ \forall r<s\leq 2r
$$
$$
\displaystyle\int_{B_{r}}|\vec{f}|^{2}|\varphi|^{2}dx\leq N_{2}^{2}r^{2\alpha_{2}}\displaystyle\int_{B_{r}}|\nabla\varphi|^{2}dx,
$$
$$
|\langle g,\varphi\rangle|\leq \frac{N_{2}r^{\alpha_{2}}}{r^{2}}|B_{r}|^{\frac{n+2}{2n}}\|\nabla\varphi\|_{L^{2}(B_{r})}.
$$
If $u\in W^{1,2}(\Omega)$ is a weak solution of ($\ref{1}$) in $\Omega$, then $u$ is $C^{\alpha}$ at 0 in the $L^{2}$ sense, where $\alpha=\min\{\alpha_{1},\alpha_{2}\}$. Furthermore, 
there are constants $C=C(n,R,\alpha_{1},\alpha_{2},N_{1})$ such that for any solution of $(\ref{1})$, there exists a constant $K$ such that for any $0<r\leq R$,
\begin{eqnarray*}
\left(\frac{1}{|B_{r}|}\displaystyle\int_{B_{r}}|u-K|^{2}dx\right)^{\frac{1}{2}}\leq
CAr^{\alpha},\quad 0<r\leq R,
\end{eqnarray*}
where
$$A=\|u\|_{L^{2}(B_{R})}+N_{2}.
$$

Especially, if $\alpha_{1}=\alpha_{2}=1$, then $u$ is continuous at 0 in the $L^{2}$ sense with the modulus of continuity $|r\ln r|$. In other words,
there are constants $C=C(n,R,N_{1})$ such that for any solution of $(\ref{1})$, there exists a constant $K$ such that for any $0<r\leq R$,
\begin{eqnarray*}
\left(\frac{1}{|B_{r}|}\displaystyle\int_{B_{r}}|u-K|^{2}dx\right)^{\frac{1}{2}}\leq
CA|r\ln r|,\quad 0<r\leq R,
\end{eqnarray*}
\end{crl}

\begin{rem}
If we replace $-\Delta u$ with $-D_{j}(a_{ij}D_{i}u)$ in $(\ref{1})$ and $a_{ij}$ satisfies uniformly elliptic condition, then Theorem $\ref{linfty}$ and the previous two corollaries still hold.
\end{rem}

We also have observed that there have some equivalences between the Dini decay conditions proposed in Theorem $\ref{linfty}$ and the necessary and sufficient conditions on the form boundedness for second order differential operators proposed by Maz'ya and Verbitsky in \cite{MV02}. The following remark says that if $\langle V\cdot,\cdot\rangle$ satisfies Dini decay condition, then we can characteristic $V=\text{div}\vec{\Gamma}_{r,y}+d^{-1}_{\partial B_{r}(y)}(x)\Gamma_{r,y,0}$ in any $B_{r}(y)\subset B_{1}$ and $\|\Gamma_i\|_{M(W^{1,2}_{0}(B_{r}(y)) \rightarrow L^2(B_{r}(y)))}$ can be controlled by the same Dini modulus of continuity. Conversely, it is also true. Besides, the remark also illustrates the reason why we make such assumption on $V$ to some extent.
\begin{rem}\label{rem5.15}Let $V\in M(W^{1,2}(B_{1}),W^{-1,2}(B_{1}))$. For any $0<r\leq\displaystyle\frac{1}{2}$ and for any $y\in B_{1}$ satisfying $B_{2r}(y)\subset B_{1}$, the following two statements are equivalent:\\
(1) For any $\psi\in W^{1,2}(B_{1})$, $\varphi\in W_{0}^{1,2}(B_{1})$ with $\text{supp}\{\varphi\}\subset \overline{B_{r}(y)}$, the following inequality holds:
\begin{equation}\label{eqv1}
|\langle V\psi,\varphi\rangle|\leq \inf_{r<s\leq2r}\omega(s)\|\psi(\cdot+y)\|_{L_{r,s}^{1,2}}\|\nabla\varphi\|_{L^{2}(B_{r}(y))}.
\end{equation}
where $\omega(r)$ is a Dini modulus of continuity satisfying $\displaystyle\int_{0}^{1}\frac{\omega(r)}{r}dr<\infty$.

(2) There exists $\vec{\Gamma}_{r,y}=(\Gamma_{r,y,1}, \cdots, \Gamma_{r,y,n})$ and $\Gamma_{r,y,0}$ such that $\Gamma_{r,y,i}\in M(W^{1,2}_{0}(B_{r}(y)) \rightarrow L^2(B_{r}(y)))$ for $0\leq i\leq n$, and $V=\text{div}\vec{\Gamma}_{r,y}+d^{-1}_{\partial B_{r}(y)}(x)\Gamma_{r,y,0}$ in $B_{r}(y)$ and
\begin{equation}\label{eqv2}
\sum_{0 \leq i \leq n}\|\Gamma_i\|_{M(W^{1,2}_{0}(B_{r}(y)) \rightarrow L^2(B_{r}(y)))} \leq C\omega(r).
\end{equation}
\end{rem}

The proof of Remark $\ref{rem5.15}$ will be given in the last section.

For the proof of Theorem $\ref{linfty}$, We will  use the perturbation technique and the compactness method. The perturbation technique can be tracked back to  \cite{Ca,Ca1989} in which they used it to prove the interior pointwise $C^{1,\alpha}$ and $C^{2,\alpha}$ regularity of the solutions for the fully nonlinear elliptic equations. The main idea is to approximate the solution by linear functions or second order polynomials in different scales like $B_{\lambda},~B_{\lambda^{2}},\cdots,~B_{\lambda^{k}},\cdots$ where $\lambda<1$ with the error as $\lambda^{k(1+\alpha)}$ or $\lambda^{k(2+\alpha)}$. In this paper, we will approximate the solution by constants and prove the sum of the error from different scales is convergent, which leads to the continuity. To achieve this goal, the key step is to get the approximation in $B_{\lambda}$ for some $\lambda<1$, then by scaling and iteration, the proof will be finished. The compactness method is inspired by \cite{BW04,W92} which is an extremely powerful tool in nonlinear analysis and will be used to prove the key lemma. The compactness method requires no solvability of Dirichlet problems, so we do not need to consider the equation of the difference of the solution and its approximation. In fact, we will not use any solvability throughout the proof.

For Theorem $\ref{thm1.2}$, Aizenman and Simon have already proved the continuity of solution when $|\vec{f}|=g=0$ in \cite{AS1982} mainly relies on the probabilistic technique and the properties of Green function. Here a new proof will be given and the method maybe seems like more PDE's. In particular, we first solve a kind of approximate equations with mollified coefficients and nonhomogeneous terms to get the smooth solutions and their uniformly estimates, then we use the fixed point theorem to prove the existence of a kind of  Dirichlet problem to get weak solutions, and finally we show the local $L^{\infty}$ estimate for a simple equation $-\Delta u+Vu=0$ by using the weak-$\ast$ convergence of a series of smooth functions with uniform $L^{\infty}$ norm.

The conditions on the coefficients in  Theorem $\ref{linfty}$ is formulated in some sense. Next, we will give some examples.

\begin{ex}\label{ex1.6}
Assume $n\geq3$. Let $\vec{f}\in L^{q}(\Omega)^{n}$, $g\in L^{\frac{nq}{n+q}}(\Omega)$, $V\in L^{\frac{q}{2}}(\Omega)$ for some $2n>q>n$. In this case, we know that the weak solution of $(\ref{1})$ belongs to $C^{\alpha}(\Omega')$ for some $0<\alpha<1$ where $\Omega'\Subset \Omega$. While, by using  Theorem $\ref{linfty}$, we also can show the above $C^{\alpha}$ regularity. In fact, there exists $0<R\leq\min\{1,\text{dist}(\Omega',\Omega)\}$ such that $B_{R}(x_{0})\subset\Omega$ for any $x_{0}\in\Omega'$. Then for any $0<r\leq \displaystyle\frac{R}{2}$ and $\psi\in W^{1,2}(\Omega)$, $\varphi\in W_{0}^{1,2}(\Omega)$ with $\text{supp}\{\varphi\}\subset\overline{B_{r}(x_{0})}$, $\langle V\psi,\varphi\rangle$, $\langle g,\varphi\rangle$ can be viewed as $\displaystyle\int_{B_{r}(x_{0})}V\phi\varphi dx$ and $\displaystyle\int_{B_{r}(x_{0})}g\varphi dx$ respectively. Using H\"{o}lder inequality, Sobolev inequality and Poincar\'{e} inequality it follows that
\begin{eqnarray*}
\left|\langle V\psi,\varphi\rangle\right|=\left|\displaystyle\int_{B_{r}(x_{0})}V\psi\varphi dx\right|&\leq&\|V\|_{L^{\frac{n}{2}}(B_{r}(x_{0}))}\|\psi\|_{L^{\frac{2n}{n-2}}(B_{r}(x_{0}))}\|\varphi\|_{L^{\frac{2n}{n-2}}(B_{r}(x_{0}))}\\
&\leq&C(n,q)r^{2-\frac{2n}{q}}\|V\|_{L^{\frac{q}{2}}(B_{r}(x_{0}))}\|\psi\|_{W^{1,2}(B_{r}(x_{0}))}\|\nabla\varphi\|_{L^{2}(B_{r}(x_{0}))}\\
&\leq&C(n,q,\|V\|_{L^{\frac{q}{2}}(\Omega)})r^{2-\frac{2n}{q}}\|\psi\|_{W^{1,2}(B_{r}(x_{0}))}\|\nabla\varphi\|_{L^{2}(B_{r}(x_{0}))}\\
&\leq&C(n,q,\|V\|_{L^{\frac{q}{2}}(\Omega)})
s^{2-\frac{2n}{q}}\|\psi(\cdot+x_{0})\|_{L^{1,2}_{r,s}}\|\nabla\varphi\|_{L^{2}(B_{r}(x_{0}))},
\end{eqnarray*}
for any $r<s\leq 2r$. Similarly,
\begin{eqnarray*}
\displaystyle\int_{B_{r}(x_{0})}|\vec{f}|^{2}\varphi^{2}dx
&\leq&C(n,q,\|\vec{f}\|_{L^{q}(\Omega)})r^{2-\frac{2n}{q}}\int_{B_{r}(x_{0})}|\nabla\varphi|^{2}dx,
\end{eqnarray*}
\begin{eqnarray*}
\left|\langle g,\varphi\rangle\right|=\left|\displaystyle\int_{B_{r}(x_{0})}g\varphi dx\right|
&\leq&C(n,q,\|g\|_{L^{\frac{nq}{n+q}}(\Omega)})r^{-1-\frac{n}{q}}|B_{r}|^{\frac{n+2}{2n}}\|\nabla\varphi\|_{L^{2}(B_{r}(x_{0}))}.
\end{eqnarray*}
We set $\omega_{1}(r)=C(n,q,\|V\|_{L^{q}(\Omega)})r^{2-\frac{2n}{q}}$, $\omega_{2}(r)=C(n,q,\|\vec{f}\|_{L^{q}(\Omega)},\|g\|_{L^{\frac{nq}{n+q}}(\Omega)})r^{1-\frac{n}{q}}$.  then by Theorem $\ref{linfty}$ or Corollary $\ref{holder}$, it follows that $u\in C^{1-\frac{n}{q}}(\Omega')$.
The special case is $q=2n$ when $|\vec{f}|=g=0$, Corollary $\ref{holder}$ implies that $u\in C(\Omega')$ with modulus of continuity $|r\ln r|$. This example shows that our theorem cover the classical H\"{o}lder regularity.
\end{ex}

\begin{ex}\label{ex1.7}
Let $\vec{f}=0$, $g=0$, $V(x)=\displaystyle\frac{1}{|x|^{2}(-\ln |x|)^{\frac{2}{n}+1}}$. It is easy to check   $V\in L^{\frac{n}{2}}(B_{r})$ for any $r<\displaystyle\frac{1}{2}$ and $V\notin L^{\frac{n}{2}+\delta}(B_{r})$ for any $\delta>0$. Furthermore,  by  a simple calculation, we have
\begin{eqnarray*}
r^{2}\left(\frac{1}{|B_{r}|}\int_{B_{r}}|V(x)|^{\frac{n}{2}}dx\right)^{\frac{2}{n}}
=r^{2}\left(\frac{1}{|B_{r}|}\int_{B_{r}}\frac{1}{|x|^{n}(-\ln |x|)^{1+\frac{n}{2}}}dx\right)^{\frac{2}{n}}=\frac{C}{-\ln r}.
\end{eqnarray*}
But for any $\hat{r}>0$, the integral $\displaystyle\int_{0}^{\hat{r}}\frac{1}{r(-\ln r)}dr$ is not convergent. Hence $V\notin C^{-2,\text{Dini}}(0)$ in $L^{\frac{n}{2}}$ sense. Then by classical result, we cannot get the continuity of $u$.

However, by using  Theorem $\ref{linfty}$, we can show the weak solution $u$ is continuous. In fact, if we set
$$\vec{h}=\frac{1}{n-2}\left(\frac{x_{1}}{|x|^{2}(-\ln |x|)^{\frac{2}{n}+1}},\frac{x_{2}}{|x|^{2}(-\ln |x|)^{\frac{2}{n}+1}},\cdots,\frac{x_{n}}{|x|^{2}(-\ln |x|)^{\frac{2}{n}+1}}\right),
$$
$$
\gamma(x)=-\frac{2+n}{n(n-2)}\frac{1}{|x|^{2}(-\ln |x|)^{\frac{2}{n}+2}},
$$
we have
 $$V(x)=\text{div}\vec{h}+\gamma.$$
Therefore, for any $\psi\in W^{1,2}(B_{1})$, $\varphi\in W_{0}^{1,2}(B_{1})$ with $\text{supp}\{\varphi\}\subset\overline{B}_{r}$, we obtain that
$$\langle V\psi,\varphi\rangle=\displaystyle\int_{B_{r}}V\psi\varphi dx=-\displaystyle\int_{B_{r}}\vec{h}\cdot\nabla(\psi\varphi) dx+\displaystyle\int_{B_{r}}\gamma\psi\varphi dx\triangleq I_{1}+I_{2}.
$$
By H\"{o}lder inequality, we can estimate
\begin{eqnarray*}
|I_{1}|&\leq&\frac{1}{(n-2)(-\ln r)^{\frac{2}{n}+1}}\int_{B_{r}}\frac{1}{|x|}(|\psi||\nabla\varphi|+|\varphi||\nabla\psi|)dx\\
&\leq&\frac{1}{(n-2)(-\ln r)^{\frac{2}{n}+1}}\left[\left(\int_{B_{r}}\frac{|\psi|^{2}}{|x|^{2}}dx\right)^{\frac{1}{2}}\left(\int_{B_{r}}|\nabla\varphi|^{2}dx\right)^{\frac{1}{2}}
+\left(\int_{B_{r}}\frac{|\varphi|^{2}}{|x|^{2}}dx\right)^{\frac{1}{2}}\left(\int_{B_{r}}|\nabla\psi|^{2}dx\right)^{\frac{1}{2}}\right],
\end{eqnarray*}
and
\begin{eqnarray*}
|I_{2}|&\leq&\frac{2+n}{n(n-2)(-\ln r)^{\frac{2}{n}+2}}\int_{B_{r}}\frac{1}{|x|^{2}}|\psi||\varphi|dx\\
&\leq&\frac{2+n}{n(n-2)(-\ln r)^{\frac{2}{n}+2}}\left(\int_{B_{r}}\frac{|\psi|^{2}}{|x|^{2}}dx\right)^{\frac{1}{2}}
\left(\int_{B_{r}}\frac{|\varphi|^{2}}{|x|^{2}}dx\right)^{\frac{1}{2}}.
\end{eqnarray*}
By Hardy inequality (given by Theorem 4 of Section 5.8 in \cite{E2010}) and Poincar\'{e} inequality, it follows that
\begin{eqnarray*}
\left(\int_{B_{r}}\frac{|\psi|^{2}}{|x|^{2}}dx\right)^{\frac{1}{2}}\leq C\left(\int_{B_{r}}\frac{|\psi|^{2}}{r^{2}}+|\nabla\psi|^{2}dx\right)^{\frac{1}{2}}\leq C\|\psi\|_{L^{1,2}_{r,s}},\quad \forall ~r<s\leq2r,
\end{eqnarray*}
\begin{eqnarray*}
\left(\int_{B_{r}}\frac{|\varphi|^{2}}{|x|^{2}}dx\right)^{\frac{1}{2}}\leq C\left(\int_{B_{r}}\frac{|\varphi|^{2}}{r^{2}}+|\nabla\varphi|^{2}dx\right)^{\frac{1}{2}}\leq C\left(\int_{B_{r}}|\nabla\varphi|^{2}dx\right)^{\frac{1}{2}}.
\end{eqnarray*}
Hence for some $r_{0}$ small enough, when $r\leq r_{0}$, we have
$$|\langle V\psi,\varphi\rangle|\leq|I_{1}|+|I_{2}|\leq \frac{C}{(-\ln s)^{\frac{2}{n}+1}}\|\psi\|_{L^{1,2}_{r,s}}\|\nabla\varphi\|_{L^{2}(B_{r})},  \quad \forall r<s\leq2r.
$$
Thus we take that $\omega_{1}(r)=\displaystyle\frac{C}{(-\ln r)^{\frac{2}{n}+1}}$. It follows that
$$\int_{0}^{r_{0}}\frac{\omega_{1}(r)}{r}dr<\infty.
$$
Therefore $u$ is continuous at 0 from Theorem $\ref{linfty}$.
\end{ex}

The remaining sections are organized as follows. In Section 2, we give some notations, some definitions which including the definition of the weak solution of $(\ref{1})$. In Section 3, we show an energy estimate and a key lemma by using the compactness method. The proof of Theorem $\ref{linfty}$ and its corollaries will be given in Section 4. In addition, Kato class condition is further considered in Section 5 and we give the proof of Theorem $\ref{thm1.2}$. In the end of this paper, we will give some additional remarks and its proof in Section 6.
\section{Notations and definition of weak solutions}
In this section, we give some notations used in this paper. Then we will give some definitions about bounded bilinear mapping, bounded linear functional and admissible measure for $W_{0}^{1,2}(\Omega)$, and introduce Kato class and $K^{\alpha}$ class. The most important part in this section is to define the weak solution of the equation.\\
%
%
%

In the sequel, we denote by $W^{-1,2}(\Omega)$ the dual space to $W_{0}^{1,2}(\Omega)$, i.e. the class of the bounded linear functional on $ W_{0}^{1,2}(\Omega)$. We write $\langle~,~\rangle$ to denote the pairing between $W^{-1,2}(\Omega)$ and $W_{0}^{1,2}(\Omega)$. Moreover, we say that $g\in W^{-1,2}(\Omega)$, i.e. $g$ is a bounded linear functional on $ W_{0}^{1,2}(\Omega)$, and $\langle g,v\rangle$ satisfies
$$
|\langle g,v\rangle|\leq C\|\nabla v\|_{L^{2}(\Omega)},
$$
for any $v\in W_{0}^{1,2}(\Omega)$, where the constant $C$ depends only on dimension $n$ and $\Omega$ but independent of $v$. We also denote $g(v)=\langle g,v\rangle$.

\begin{df}\label{fbv} For a linear operator $V$ from $W^{1,2}(\Omega)$ to $W^{-1,2}(\Omega)$, if the bilinear mapping$$\langle V\cdot,\cdot\rangle:~W^{1,2}(\Omega)\times W_{0}^{1,2}(\Omega)\rightarrow \mathbb{R}$$
is bounded, i.e. there exists a constant $C$ depends only on dimension $n$ and $\Omega$ such that
\begin{equation}\label{vf1}
|\langle Vu,v\rangle|\leq C\|u\|_{W^{1,2}(\Omega)}\|\nabla v\|_{L^{2}(\Omega)}.
\end{equation} for any $u\in W^{1,2}(\Omega)$ and  $v\in W_{0}^{1,2}(\Omega)$, we call this $V$ as a bounded linear multiplier  from $W^{1,2}(\Omega)$ to $W^{-1,2}(\Omega)$. We also write $M(W^{1,2}(\Omega)\rightarrow W^{-1,2}(\Omega))$ to denote the class of the bounded linear multipliers from $W^{1,2}(\Omega)$ to $W^{-1,2}(\Omega)$.
\end{df}

\begin{rem}In virtue of the norm of an operator, if $V\in M(W^{1,2}(\Omega)\rightarrow W^{-1,2}(\Omega))$, then we can write that
$$\|V(x)\|_{M(W^{1,2}(\Omega)\rightarrow W^{-1,2}(\Omega))}\leq C.$$
Compared with \cite{MV02}, we remove the restriction $\langle Vu,v\rangle=\langle V,uv\rangle$ in this paper.
\end{rem}

Next we introduce  a class of admissible measure for $W_{0}^{1,2}(\Omega)$ which is a local version of the admissible measure for $W_0^{1,2}(\mathbb{R}^{n})$ mentioned in the introduction. We refer the readers to see \cite{MV02,MV06} for more details.
\begin{df}\label{dm12}
We say a nonnegative Borel measure $\mu$ on $\Omega$ belongs to the class of admissible measures for $W_{0}^{1,2}(\Omega)$, if $\mu$ obeys the trace inequality,
\begin{equation}\label{m1}
\displaystyle\int_{\Omega}|\varphi|^{2}d\mu\leq C\|\nabla\varphi\|_{L^{2}(\Omega)}^{2},\quad \forall~\varphi\in W_{0}^{1,2}(\Omega),
\end{equation}
where the constant $C$ only depends on $n$ and $\Omega$ but does not depend on $\varphi$. We also write $M(W_0^{1,2}(\Omega)\rightarrow L^{2}(\Omega))$ to denote the class of admissble measures from $W_0^{1,2}(\Omega)$ to $L^{2}(\Omega)$. Especially, for admissible measures $q(x)dx$ with nonnegative density $q\in L^{1}(\Omega)$, we will write $(\ref{m1})$ as
\begin{eqnarray*}
\displaystyle\int_{\Omega}|\varphi|^{2}q(x)dx\leq C\|\nabla\varphi\|_{L^{2}(\Omega)}^{2},\quad \forall~\varphi\in W_{0}^{1,2}(\Omega).
\end{eqnarray*} Similarly, if $q(x)\in M(W_0^{1,2}(\Omega)\rightarrow L^{2}(\Omega))$, then we can write that
$$\|q(x)\|_{M(W_0^{1,2}(\Omega)\rightarrow L^{2}(\Omega))}\leq C.$$
\end{df}

Based on above definitions, we can define the weak solutions of equation $(\ref{1})$.
\begin{df}\label{ws1} Let $\Omega$ be a bounded domain. We assume $V\in M(W^{1,2}(\Omega)\rightarrow W^{-1,2}(\Omega))$, $\vec{f}\in L^{2}(\Omega)^{n}$ satisfying that $|\vec{f}|^2$ is an admissible measure for $W_{0}^{1,2}(\Omega)$, and $g\in W^{-1,2}(\Omega)$.
We say $u\in W^{1,2}(\Omega)$ is a weak solution of
$$-\Delta u+Vu=-\text{div}\vec{f}+g\quad \text{in}~\Omega,
$$
if for all $\varphi\in W_{0}^{1,2}(\Omega)$, we have
$$\displaystyle\int_{\Omega}\nabla u\cdot \nabla\varphi dx+\langle Vu,\varphi\rangle=\displaystyle\int_{\Omega}\vec{f}\cdot \nabla\varphi dx+\langle g,\varphi\rangle.
$$
\end{df}

\begin{rem}By Definition $\ref{fbv}$, we know that $\langle Vu,\varphi\rangle$ is well-defined. It is also easy to that the other terms in the definition of the weak solution are well defined. It can be seen that $\vec{f}\in L^{2}(\Omega)^{n}$ is enough to define the weak solution. Here we also assume that $|\vec{f}|^2$ is an admissible measure for $W_{0}^{1,2}(\Omega)$, which is a basic assumption when we study in the sequel the continuity of weak solutions.
\end{rem}

Besides, the weak solution of $(\ref{1})$ can also be defined under Kato class type assumptions on $V,~\vec{f},~g$. Next we define $K^{\alpha}$ class and introduce the special case: Kato class.
\begin{df}[$K^{\alpha}$ class]\label{ka}
Let $\Omega$ be a bounded domain in $\mathbb{R}^{n}(n\geq3)$. For $\alpha >0$, we say that a function $V:\Omega\rightarrow \mathbb{R}$ belongs to $K^{\alpha}$ class with  modulus of continuity $\eta(r)$ satisfying $\lim\limits_{r\rightarrow0}\eta(r)=0$, denoting $V\in K^{\alpha}_{\eta}(\Omega)$, if
$$\sup\limits_{x\in\Omega}\int_{\Omega\cap B_{r}(x)}\frac{|V(y)|}{|x-y|^{n-\alpha}}dy\leq\eta(r).
$$
And we denote
$$\|V\|_{K^{\alpha}(\Omega)}=\sup\limits_{x\in\Omega}\int_{\Omega}\frac{|V(y)|}{|x-y|^{n-\alpha}}dy.
$$
Obviously, if $V\in K^{\alpha}_{\eta}(\Omega)$, then $V\in L^{1}(\Omega)$.

Especially, when $\alpha=2$, we say $V$ belongs to Kato class. For convenience, we denote $K^{2}_{\eta}(\Omega)$ by $K_{\eta}(\Omega)$ and
$$\|V\|_{K(\Omega)}=\sup\limits_{x\in\Omega}\int_{\Omega}\frac{|V(y)|}{|x-y|^{n-2}}dy.
$$
\end{df}
\begin{rem}
For any $p>\displaystyle\frac{n}{2}$, $V\in L^{p}(\Omega)$ can guarantee that $V\in K(\Omega)$. But $L^{\frac{n}{2}}(\Omega)$ is incomparable with $K(\Omega)$ for $n\geq3$.
\end{rem}
\begin{rem}\label{rem2.9}

By definition and Theorem 4.15 in \cite{AS1982}, if $V\in K_{\eta}(\Omega)$, then the function
$$h(x)=\int_{\Omega}\frac{|V(y)|}{|x-y|^{n-2}}dy
$$
is well defined for all $x\in\Omega$, and $h(x)$ is a continuous function in $\Omega$.
\end{rem}


For Kato class, the following property can be found in \cite{CGL1993}.
\begin{lm}\label{lm5.4}
Assume $V\in K_{\eta}(\Omega)$. If $u\in W_{\emph{loc}}^{1,2}(\Omega)$, $x_{0}\in \Omega$ and $\overline{B_{R+2r}(x_{0})}\subseteq\Omega$, then
$$\int_{B_{R}(x_{0})}|V(x)|u^{2}(x)dx\leq C\left(\sup_{x\in B_{R+2r}(x_{0})}\int_{B_{2r}(x)}\frac{|V(y)|}{|x-y|^{n-2}}dy\right)
\left(\frac{1}{r^{2}}\|u\|_{L^{2}(B_{R+r}(x_{0}))}^{2}+\|\nabla u\|_{L^{2}(B_{R+r}(x_{0}))}^{2}\right).
$$
\end{lm}

Then we can define the weak solution of $(\ref{1})$ in general case, since it is easy to know the integral $\displaystyle\int_{B_{R}}Vu\varphi dx$ and $\displaystyle\int_{B_{R}}g\varphi dx$ is finite based on this lemma.
\begin{df}
$\Omega$ is a bounded domain. We say $u\in W^{1,2}(B_{R})$ is a weak solution of $(\ref{1})$ in $B_{R}\subset\Omega$ for $V\in K_{\eta_{V}}(B_{1})$, $\vec{f}\in K^{1}_{\eta_{f}}(B_{1})\cap L^{2}(B_{1})^{n}$, $g\in K_{\eta_{g}}(B_{1})$, if for all $\varphi\in W_{0}^{1,2}(\Omega)$, we have
$$\displaystyle\int_{B_{R}}\nabla u\cdot \nabla\varphi +Vu\varphi dx=\displaystyle\int_{B_{R}}\vec{f}\cdot \nabla\varphi +g\varphi dx.
$$
\end{df}
For more properties on Kato class and $K^{1}$ class, we will introduce in Section 5.

\section{Energy estimate and compactness lemma}
In this section we will give some preliminary lemmas to prove Theorem $\ref{linfty}$ including energy estimates, and an approximation lemma. By Definition \ref{ws1}, we assume in the sequel that $V\in M(W^{1,2}(\Omega),W^{-1,2}(\Omega))$, $\vec{f}\in L^{2}(\Omega)^{n}$ satisfying that $|\vec{f}|^2$ is an admissible measure for $W_{0}^{1,2}(\Omega)$, and $g\in W^{-1,2}(\Omega)$. The following lemma will be used in the energy estimate which can be found in Chapter 4 of \cite{HL2011}.
\begin{lm}\label{pee}
Let $h(t)\geq0$ be bounded in $[\tau_{0},\tau_{1}]$ with $\tau_{0}\geq0$. Suppose for any $\tau_{0}\leq t<s \leq\tau_{1}$,
$$h(t)\leq \theta h(s)+\frac{A}{(s-t)^{\alpha}}+B,
$$
for some $\theta\in[0,1)$ and some $A,B\geq0$. Then for any $\tau_{0}\leq t<s \leq\tau_{1}$,
$$h(t)\leq C\left(\frac{A}{(s-t)^{\alpha}}+B\right),
$$
where $C$ is a positive constant depending only on $\alpha$ and $\theta$.
\end{lm}

We need the following energy estimate for the proof of Lemma $\ref{lm4.2}$.
\begin{lm}[Energy~estimate]\label{ee}Assume there exists a sufficiently small positive constant $\varepsilon_{0}$
such that for any $0<\rho\leq\displaystyle\frac{1}{2}$, $\psi\in W^{1,2}(B_{1}),~\varphi\in W_{0}^{1,2}(B_{1})$ with $\text{supp}\{\varphi\}\subset \overline{B}_{\rho}$,
\begin{equation}\label{v1}
|\langle V\psi,\varphi\rangle|\leq \varepsilon_{0}\|\psi\|_{L_{\rho,\theta}^{1,2}}\|\nabla\varphi\|_{L^{2}(B_{\rho})}, ~~~\forall \rho<\theta\leq2\rho,
\end{equation}
\begin{equation}\label{g1}
|\langle g,\varphi\rangle|
\leq \varepsilon_{0}\|\nabla\varphi\|_{L^{2}(B_{\rho})},
\end{equation}
\begin{equation}\label{h22}
\displaystyle\int_{B_{\rho}}|\vec{f}|^{2}|\varphi|^{2}dx\leq \varepsilon_{0}\displaystyle\int_{B_{\rho}}|\nabla\varphi|^{2}dx.
\end{equation}
Then if $u\in W^{1,2}(B_{1})$ is a weak solution of
\begin{eqnarray*}
-\Delta u+Vu=-\emph{div}\vec{f}+g\quad \text{in}~B_{1},
\end{eqnarray*}
we have for any $0<t<s\leq \displaystyle\frac{1}{2}$,
$$\displaystyle\int_{B_{t}}|\nabla u|^{2}dx\leq C\left(\displaystyle\int_{B_{1}}|u|^{2}dx+1\right)\frac{1}{(s-t)^{2}}+1,
$$
where $C$ is a positive constant depending only on $n$ and $\varepsilon_{0}$.
\end{lm}
\begin{proof}
First we take $\eta\in C_{0}^{\infty}(B_{1})$ with $0\leq\eta\leq1$ in $B_{1}$, $\eta=1$ in $B_{t}$, $\eta=0$ in $B_{1}\backslash B_{\frac{t+s}{2}}$, $|\nabla\eta|\leq\displaystyle\frac{c_{0}}{s-t}$.
Note that $\eta^{2}u\in W_{0}^{1,2}(B_{1})$, then by Definition $\ref{ws1}$, we have
$$\displaystyle\int_{B_{1}}\nabla u\cdot \nabla(\eta^{2}u)dx+\langle Vu,\eta^{2}u\rangle=\displaystyle\int_{B_{1}}\vec{f}\cdot \nabla(\eta^{2}u)dx+\langle g,\eta^{2}u\rangle.
$$
We rewrite the above expression as
$$I_{1}=I_{2}+I_{3}+I_{4}+I_{5},
$$
where
$$I_{1}=\displaystyle\int_{B_{1}}\eta^{2}\nabla u\cdot \nabla udx=\displaystyle\int_{B_{1}}\eta^{2}|\nabla u|^{2}dx,
$$
$$I_{2}=\displaystyle\int_{B_{1}}(2u\eta\vec{f}\cdot \nabla\eta+\eta^{2}\vec{f}\cdot\nabla u)dx,
$$
$$I_{3}=\langle g,\eta^{2}u\rangle,
$$
$$I_{4}=-\displaystyle\int_{B_{1}}2u\eta\nabla u\cdot \nabla \eta dx,
$$
$$I_{5}=-\langle Vu,\eta^{2}u\rangle.
$$
Since $\text{supp}\{\eta\}\subset \overline{B}_{\frac{t+s}{2}}\subset \overline{B}_{s}$, then by using the Cauchy inequality with $\tau<1$ and the assumption $(\ref{v1}),(\ref{g1}),(\ref{h22})$,  we obtain
\begin{eqnarray*}
|I_{2}| &\leq&\displaystyle\int_{B_{1}}(2|\eta\vec{f}||u\nabla\eta|+\eta^{2}|\vec{f}||\nabla u|)dx\\
&\leq&\displaystyle\int_{B_{1}}\left(\eta^{2}|\vec{f}|^{2}+|u|^{2}|\nabla \eta|^{2}+\tau\eta^{2}|\nabla u|^{2}+\frac{1}{4\tau}\eta^{2}|\vec{f}|^{2}\right)dx\\
&\leq&\left(1+\frac{1}{4\tau}\right)\varepsilon_{0}\displaystyle\int_{B_{1}}|\nabla \eta|^{2}dx+
\displaystyle\int_{B_{1}}|u|^{2}|\nabla \eta|^{2}dx
+\tau\displaystyle\int_{B_{1}}\eta^{2}|\nabla u|^{2}dx\\
&\leq&\tau\displaystyle\int_{B_{s}}|\nabla u|^{2}dx+\left(\left(1+\frac{1}{4\tau}\right)\varepsilon_{0}|B_{1}|
+\int_{B_{1}}|u|^{2}dx\right)\frac{c_{0}^{2}}{(s-t)^{2}},
\end{eqnarray*}

\begin{eqnarray*}
|I_{3}|&=&|\langle g,\eta^{2}u\rangle|\\
&\leq&\varepsilon_{0}\|\nabla(\eta^{2}u)\|_{L^{2}(B_{s})}\\
&\leq&\varepsilon_{0}\|\eta^{2}\nabla u+2\eta u\nabla\eta\|_{L^{2}(B_{1})}\\
&\leq&\varepsilon_{0}\left(2\displaystyle\int_{B_{1}}\eta^{4}|\nabla u|^{2}dx+8\displaystyle\int_{B_{1}}\eta^{2}|u|^{2}|\nabla \eta|^{2}dx\right)^{\frac{1}{2}}\\
&\leq&\varepsilon_{0}\left(2\displaystyle\int_{B_{1}}\eta^{2}|\nabla u|^{2}dx+8\displaystyle\int_{B_{1}}|u|^{2}|\nabla \eta|^{2}dx\right)^{\frac{1}{2}}\quad \text{since}~~ 0\leq\eta\leq1\\
&\leq&\frac{\varepsilon_{0}}{2}\left(1+2\displaystyle\int_{B_{1}}\eta^{2}|\nabla u|^{2}dx+8\displaystyle\int_{B_{1}}|u|^{2}|\nabla \eta|^{2}dx\right)\\
&\leq&\varepsilon_{0}\displaystyle\int_{B_{s}}|\nabla u|^{2}dx+\left(4\varepsilon_{0}\int_{B_{1}}|u|^{2}dx\right)\frac{c_{0}^{2}}{(s-t)^{2}}+\frac{\varepsilon_{0}}{2},
\end{eqnarray*}

\begin{eqnarray*}
|I_{4}| &\leq& \displaystyle\int_{B_{1}}2|u\nabla \eta||\eta\nabla u|dx \leq \frac{1}{\tau}\displaystyle\int_{B_{1}}|u|^{2}|\nabla \eta|^{2}dx
+\tau\displaystyle\int_{B_{1}}\eta^{2}|\nabla u|^{2}dx\\
&\leq&\tau\displaystyle\int_{B_{s}}|\nabla u|^{2}dx+\left(\frac{1}{\tau}\displaystyle\int_{B_{1}}|u|^{2}dx\right)\frac{c_{0}^{2}}{(s-t)^{2}},
\end{eqnarray*}
and
\begin{eqnarray*}
|I_{5}|=|\langle Vu,\eta^{2}u\rangle|&\leq& \varepsilon_{0}\|u\|_{L_{\frac{t+s}{2},s}^{1,2}}\|\nabla(\eta^{2} u)\|_{L^{2}(B_{\frac{t+s}{2}})}\\
&\leq&\frac{\varepsilon_{0}}{2}\left(\|u\|_{L_{\frac{t+s}{2},s}^{1,2}}^{2}+\|\nabla(\eta^{2} u)\|_{L^{2}(B_{1})}^{2}\right)\\
&\leq&\frac{\varepsilon_{0}}{2}\left(2\int_{B_{s}}(\frac{4u^{2}}{(s-t)^{2}}+|\nabla u|^{2})dx+2\int_{B_{1}}(\eta^{4}|\nabla u|^{2}+4\eta^{2}u^{2}|\nabla\eta|^{2})dx\right)\\
&\leq&\frac{\varepsilon_{0}}{2}\left(4\int_{B_{s}}|\nabla u|^{2}dx+2\int_{B_{1}}(\frac{4u^{2}}{(s-t)^{2}}+4u^{2}|\nabla\eta|^{2})dx\right)\\
&\leq&2\varepsilon_{0}\int_{B_{s}}|\nabla u|^{2}dx+\left(4\varepsilon_{0}(c_{0}^{2}+1)\displaystyle\int_{B_{1}}|u|^{2}dx\right)\frac{1}{(s-t)^{2}}.
\end{eqnarray*}
Now we combine the estimates $I_{i}(i=1,2,3,4,5)$ to yield that
\begin{eqnarray*}
\displaystyle\int_{B_{t}}|\nabla u|^{2}dx &\leq& \displaystyle\int_{B_{1}}\eta^{2}|\nabla u|^{2}dx = I_{1}\\
&\leq&(3\varepsilon_{0}+2\tau)\int_{B_{s}}|\nabla u|^{2}dx
+\left((1+\frac{1}{\tau}+8\varepsilon_{0})\displaystyle\int_{B_{1}}|u|^{2}dx+(1+\frac{1}{4\tau})\varepsilon_{0}|B_{1}|\right)
\frac{c_{0}^{2}}{(s-t)^{2}}\\
&&+\left(4\varepsilon_{0}\int_{B_{1}}u^{2}dx\right)\frac{1}{(s-t)^{2}}+\frac{\varepsilon_{0}}{2}.
\end{eqnarray*}
We choose $\tau$ small enough such that $3\varepsilon_{0}+2\tau\leq\displaystyle\frac{1}{2}$, then by Lemma $\ref{pee}$, it follows that $0\leq t<s \leq\displaystyle\frac{1}{2}$,
\begin{eqnarray*}\displaystyle\int_{B_{t}}|\nabla u|^{2}dx&\leq& C\left(\left(((1+\frac{1}{\tau}+8\varepsilon_{0})c_{0}^{2}+4\varepsilon_{0})\displaystyle\int_{B_{1}}|u|^{2}dx
+(1+\frac{1}{4\tau})c_{0}^{2}\varepsilon_{0}|B_{1}|\right)
\frac{1}{(s-t)^{2}}+\frac{\varepsilon_{0}}{2}\right)\\
&\leq&C\left(\displaystyle\int_{B_{1}}|u|^{2}dx+1\right)\frac{1}{(s-t)^{2}}+1,
\end{eqnarray*}
for a positive constant $C$ depending only on $n$ and $\varepsilon_{0}$.
\end{proof}

Next we show the following approximation lemma by the compactness method.
\begin{lm}\label{lm4.2}
For any $\varepsilon>0$, there exists a small $\delta=\delta(\varepsilon)>0$ such that for any weak solution of
\begin{eqnarray*}
-\Delta u+Vu=-\emph{div}\vec{f}+g\quad \text{in}~B_{1}
\end{eqnarray*}
with $\displaystyle\frac{1}{|B_{1}|}\displaystyle\int_{B_{1}}u^{2}dx\leq 1$,  and for any $0<\rho\leq\displaystyle\frac{1}{2}$, $\psi\in W^{1,2}(B_{1}),~\varphi\in W_{0}^{1,2}(B_{1})$ with $\text{supp}\{\varphi\}\subset \overline{B}_{\rho}$,
$$|\langle V\psi,\varphi\rangle|\leq \delta \|\psi\|_{L_{\rho,\theta}^{1,2}}\|\nabla\varphi\|_{L^{2}(B_{\rho})},~~~\forall \rho<\theta\leq2\rho,$$
$$\displaystyle\int_{B_{\rho}}|\vec{f}|^{2}\varphi^{2}dx\leq\delta^{2}\displaystyle\int_{B_{\rho}}|\nabla \varphi|^{2}dx,
$$
$$|\langle g,\varphi\rangle|\leq \delta\|\nabla\varphi\|_{L^{2}(B_{\rho})},$$
there exists a harmonic function $p(x)$ defined in $B_{\frac{1}{2}}$ such that
$$\int_{B_{\frac{1}{8}}}|u-p|^{2}dx\leq\varepsilon^{2}.
$$
\end{lm}
\begin{proof}
We prove it by contradiction. Suppose that there exists $\bar{\varepsilon}>0$, $u_{k},~\vec{f}_{k}$, $g_{k}$ and $V_{k}$ where $\langle V_{k}\cdot,\cdot\rangle$ is bounded on $W^{1,2}(B_{1})\times W_{0}^{1,2}(B_{1})$,
$|\vec{f}_{k}|^{2}$ is an admissible measure on $W^{1,2}(B_{1})$, $g_{k}$ is a bounded linear functional on $W_{0}^{1,2}(B_{1})$, $u_k$ satisfies
$$-\Delta u_{k}+V_{k}u_{k}=-\text{div}\vec{f}_{k}+g_{k}, \quad \text{weakly in}~B_{1},
$$
$$\frac{1}{|B_{1}|}\displaystyle\int_{B_{1}}u_{k}^{2}dx\leq 1,
$$
and for any $0<\rho\leq\displaystyle\frac{1}{2}$, $\psi\in W^{1,2}(B_{1}),~\varphi\in W_{0}^{1,2}(B_{1})$ with $\text{supp}\{\varphi\}\subset \overline{B}_{\rho}$,
\begin{equation}\label{vn1}
|\langle V_{k}\psi,\varphi\rangle|\leq \frac{1}{k}\|\psi\|_{L_{\rho,\theta}^{1,2}}\|\nabla\varphi\|_{L^{2}(B_{\rho})},~~~ \forall \rho<\theta\leq2\rho,
\end{equation}
\begin{equation}\label{g2}
|\langle g_{k},\varphi\rangle|\leq \frac{1}{k}\|\nabla\varphi\|_{L^{2}(B_{\rho})},
\end{equation}
\begin{equation}\label{f1}
\displaystyle\int_{B_{\rho}}|\vec{f}_{k}|^{2}\varphi^{2}dx\leq\frac{1}{k^{2}}\displaystyle\int_{B_{\rho}}|\nabla \varphi|^{2}dx,
\end{equation}
such that for any harmonic function $p$ defined in $B_{\frac{1}{8}}$ ,
\begin{equation}\label{j-1}\int_{B_{\frac{1}{8}}}|u_{k}-p|^{2}dx\geq\bar{\varepsilon}^{2}.
\end{equation}
If letting $k$ be large such that $\displaystyle\frac{1}{k}+\displaystyle\frac{1}{k^{2}}<\varepsilon_{0}$ and taking $t=\displaystyle\frac{1}{4}$, $s=\displaystyle\frac{1}{2}$, then we obtain from  Lemma $\ref{ee}$ that
$$\displaystyle\int_{B_{\frac{1}{4}}}|\nabla u_{k}|^{2}dx\leq C\left(\displaystyle\int_{B_{1}}|u_{k}|^{2}dx+1\right)\leq C\left(|B_{1}|+1\right)\leq C.
$$
Hence, $\{u_{k}\}$ has a subsequence, we still denote it by $\{u_{k}\}$, such that
$$u_{k}\rightharpoonup u\quad\text{in}~H^{1}(B_{\frac{1}{8}}),
$$
$$u_{k}\rightarrow u\quad\text{in}~L^{2}(B_{\frac{1}{8}}).
$$
Then we will show that $u$ itself is harmonic in $B_{\frac{1}{8}}$, which is a contradiction. In fact, for any test function $\eta\in H_{0}^{1}(B_{\frac{1}{8}})$, we extend $\eta=0$ in $B_{1}\backslash B_{\frac{1}{8}}$, still denoted by $\eta$. Since $u_k$ is a weak solution, we have
\begin{equation}\label{nc}
\displaystyle\int_{B_{\frac{1}{8}}}\nabla u_{k}\cdot \nabla\eta dx+\langle V_{k}u_{k},\eta\rangle=\displaystyle\int_{B_{\frac{1}{8}}}\vec{f}_{k}\cdot \nabla\eta dx+\langle g_{k},\eta \rangle.
\end{equation}
In $(\ref{f1})$, if we take $\varphi\in C_{0}^{\infty}(B_{1})$ with $\varphi=1$ in $B_{\frac{1}{4}}$, $\varphi=0$ in $B_{1}\backslash B_{\frac{1}{2}}$, $0\leq\varphi\leq1$ in $B_{1}$, we yield that
$$\displaystyle\int_{B_{\frac{1}{4}}}|\vec{f}_{k}|^{2}dx\leq\frac{1}{k^{2}}\displaystyle\int_{B_{1}}|\nabla \varphi|^{2}dx\leq \frac{C}{k^{2}},
$$
where $C$ is a universal constant.
Hence by H\"{o}lder inequality, we have
\begin{eqnarray*}
\left|\displaystyle\int_{B_{\frac{1}{8}}}\vec{f}_{k}\cdot \nabla\eta dx\right| \leq
\left(\displaystyle\int_{B_{\frac{1}{8}}}|\vec{f}_{k}|^{2}dx\right)^{\frac{1}{2}}
\left(\displaystyle\int_{B_{\frac{1}{8}}}|\nabla \eta|^{2}dx\right)^{\frac{1}{2}}
\leq \frac{C}{k}\rightarrow0,~~\text{as}~k\rightarrow\infty.
\end{eqnarray*}
Next we apply $(\ref{vn1})$ and $(\ref{g2})$ to obtain
$$|\langle V_{k}u_{k},\eta\rangle|\leq \frac{1}{k}\|u_{k}\|_{L_{\frac{1}{8},\frac{1}{4}}^{1,2}}\|\nabla\eta\|_{L^{2}(B_{\frac{1}{8}})}\leq \frac{C}{k}\rightarrow0,~~\text{as}~k\rightarrow\infty,
$$
$$|\langle g_{k},\eta\rangle|\leq \frac{1}{k}\|\nabla\eta\|_{L^{2}(B_{\frac{1}{8}})}\leq \frac{C}{k}\rightarrow0,~~\text{as}~k\rightarrow\infty.
$$
Now letting  $k\rightarrow\infty$ in $(\ref{nc})$, we have  that
$$\displaystyle\int_{B_{\frac{1}{8}}}\nabla u\cdot \nabla\eta dx=0.
$$
Thus we yield a harmonic function $u$ in $B_{\frac {1}{8}}$, which is contradict to (\ref{j-1}).
\end{proof}

In the following we give a key lemma, which will be used repeatedly in next section.
\begin{lm}[Key Lemma]\label{key1}
There exists $C_{0},~0<\lambda<1,~\delta_{0}>0$ such that for any weak solution of
$$-\Delta u+Vu=-\emph{div}\vec{f}+g\quad\text{in}~B_{1}$$
with
$$\displaystyle\frac{1}{|B_{1}|}\displaystyle\int_{B_{1}}u^{2}dx\leq 1,
$$
and
$$|\langle V\psi,\varphi\rangle|\leq \delta_{0}\|\psi\|_{L_{\rho,\theta}^{1,2}}\|\nabla\varphi\|_{L^{2}(B_{\rho})},~~~\forall \rho<\theta\leq2\rho,$$
$$|\langle g,\varphi\rangle|\leq \delta_{0}\|\nabla\varphi\|_{L^{2}(B_{\rho})},$$
$$\displaystyle\int_{B_{\rho}}|\vec{f}|^{2}\varphi^{2}dx\leq\delta_{0}^{2}\displaystyle\int_{B_{\rho}}|\nabla \varphi|^{2}dx,
$$
for any $0<\rho\leq\displaystyle\frac{1}{2}$, $\psi\in W^{1,2}(B_{1})$ and $\varphi\in W_{0}^{1,2}(B_{1})$ with $\text{supp}\{\varphi\}\subset \overline{B}_{\rho}$,
there exists a linear function $l(x)=a+\vec{b}\cdot x$ such that
$$\left(\frac{1}{|B_{\lambda}|}\int_{B_{\lambda}}|u-l(x)|^{2}dx\right)^{\frac{1}{2}}\leq\frac{1}{2}\lambda,
$$
and
$$|a|+|\vec{b}|\leq C_{0}.
$$
\end{lm}
\begin{proof}
Let $p$ be the harmonic function of the previous lemma which satisfies that
$$
\int_{B_{\frac{1}{8}}}|u-p|^{2}dx\leq \varepsilon^2
$$
for some $\varepsilon<1$ to be determined. Hence we have
$$\int_{B_{\frac{1}{8}}}|p|^{2}dx\leq2\int_{B_{1}}|u|^{2}dx+2\int_{B_{\frac{1}{8}}}|u-p|^{2}dx\leq2|B_{1}|+2\varepsilon^{2}\leq 4|B_{1}|.
$$
By the properties of harmonic functions, we can choose a constant $C_0$ such that for $|x|\leq\displaystyle\frac{1}{16}$,
$$|\nabla^{2}p(x)|+|\nabla p(x)|+|p(x)|\leq C\int_{B_{\frac{1}{8}}}|p|^{2}dx\leq C_{0}.
$$
Now, we take $l(x)=p(0)+\nabla p(0)\cdot x$. Then there exists $|\xi|\leq\displaystyle\frac{1}{16}$ such that
$$|p(x)-l(x)|=|\nabla^{2}p(\xi)||x|^{2}\leq C_{0}|x|^{2},\quad |x|\leq\frac{1}{16}.
$$
Therefore for each $0<\lambda<\displaystyle\frac{1}{16}$, we have
\begin{eqnarray*}
\frac{1}{|B_{\lambda}|}\int_{B_{\lambda}}|u-l(x)|^{2}dx&\leq&\frac{2}{|B_{\lambda}|}\int_{B_{\lambda}}|u-p|^{2}dx
+\frac{2}{|B_{\lambda}|}\int_{B_{\lambda}}|p-l(x)|^{2}dx\\
&\leq&\frac{2\varepsilon^{2}}{|B_{\lambda}|}+2C_{0}^{2}\lambda^{4}.
\end{eqnarray*}
Now we take $\lambda$ small enough such that
$$2C_{0}^{2}\lambda^{4}\leq\frac{1}{8}\lambda^{2},
$$
and further we take $\varepsilon$ small sufficiently such that
$$\frac{2\varepsilon^{2}}{|B_{\lambda}|}\leq\frac{1}{8}\lambda^{2}
$$
and  $\delta_{0}=\delta(\varepsilon)$ in Lemma $\ref{lm4.2}$. Thus  Lemma follows.
\end{proof}

\section{Continuity of solution under Dini decay conditions}
In this section, we will prove Theorem $\ref{linfty}$. We divide the proof into the following five steps. The first step is to normalize the problem. Secondly we use the key lemma repeatedly to give a iteration result, that is, we approximate the solution by linear functions in different scales. The next step is to prove the sum of errors from each scale is convergent. Finally we find a fixed constant to approximate $u$ in different scales. Scaling back, the continuity of solution follows.\\
\\
{\bf{Proof of Theorem $\ref{linfty}$:}}\\
\
\

$\bf{Step~1: Normalization}$\\

Firstly we give the normalization of the estimates. By the property of Dini modulus of continuity, there exists $\hat{r}$ small enough such that
\begin{equation}\label{4-1}
C_{0}|B_{1}|^{\frac{n+2}{2n}}
\left(\omega_{1}(\hat{r}R)+\frac{1}{1-\lambda}\int_{0}^{\hat{r}R}\frac{\omega_{1}(s)}{s}ds\right)\leq\frac{\delta_{0}}{128},
\end{equation}
where $\delta_{0}$, $\lambda$ and $C_{0}$ are the constants in Lemma $\ref{key1}$. In the following we denote $\bar{r}=\hat{r}R$. Hence $\bar{r}\leq1$. 

We use the nonlinear method to realize the normalization. We set
$$w(x)=\displaystyle\frac{u(\bar{r}x)}{\left(\displaystyle\frac{1}{|B_{\bar{r}}|}\int_{B_{R}}u^{2}dx
\right)^{\frac{1}{2}}+\displaystyle\frac{4|B_{1}|^{\frac{n+2}{2n}}}{\delta_{0}}\left(\omega_{2}(R)
+\displaystyle\frac{1}{1-\lambda}\displaystyle\int_{0}^{R}\frac{\omega_{2}(s)}{s}ds\right)}
\triangleq\displaystyle\frac{u(\bar{r}x)}{A},\quad x\in B_{1},
$$
$$\tilde{\omega}_{1}(r)=\omega_{1}(\bar{r}r),\quad \tilde{\omega}_{2}(r)=\omega_{2}(\bar{r}r).
$$
Then $w(x)$ is a weak solution of
\begin{equation}\label{4-2}-\Delta w+V_{\bar{r}}w=-\text{div}\vec{f}_{\bar{r}}+g_{\bar{r}}\quad \text{in}~B_{1},
\end{equation}
where $\vec{f}_{\bar{r}}(x)=\displaystyle\frac{\bar{r}}{A}\vec{f}(\bar{r}x)$, $\langle V_{\bar{r}}\cdot,\cdot\rangle$ is bounded on $W^{1,2}(B_{1})\times W_{0}^{1,2}(B_{1})$ satisfying
$$\langle V_{\bar{r}}\psi,\varphi\rangle=\frac{\bar{r}^{2}}{\bar{r}^{n}}\langle V\tilde{\psi},\tilde{\varphi}\rangle
$$
for $\tilde{\psi}(x)=\psi(\displaystyle\frac{x}{\bar{r}})\in W^{1,2}(B_{\bar{r}}),~\tilde{\varphi}(x)=
\varphi(\displaystyle\frac{x}{\bar{r}})\in W_{0}^{1,2}(B_{\bar{r}})$, and  $g_{\bar{r}}$ is a bounded linear functional on $W_{0}^{1,2}(B_{1})$ satisfying
$$\langle g_{\bar{r}},\varphi\rangle=\frac{\bar{r}^{2}}{A\bar{r}^{n}}\langle g,\tilde{\varphi}\rangle
$$
for $\tilde{\varphi}(x)=
\varphi(\displaystyle\frac{x}{\bar{r}})\in W_{0}^{1,2}(B_{\bar{r}})$.

Then it follows that
\begin{equation}\label{4-3}\left(\frac{1}{|B_{1}|}\displaystyle\int_{B_{1}}w^{2}dx\right)^{\frac{1}{2}}=
\frac{\left(\displaystyle\frac{1}{|B_{1}|}\displaystyle\int_{B_{1}}u(\bar{r}x)^{2}dx\right)^{\frac{1}{2}}}
{A}
\leq\frac{\left(\displaystyle\frac{1}{|B_{\bar{r}}|}\displaystyle\int_{B_{\bar{r}}}u^{2}dx\right)^{\frac{1}{2}}}
{\left(\displaystyle\frac{1}{|B_{\bar{r}}|}\int_{B_{R}}u^{2}dx\right)^{\frac{1}{2}}
}\leq1.
\end{equation}
Moreover, for any $0<r\leq\displaystyle\frac{1}{2}$, we have for any $r<s\leq 2r$
\begin{eqnarray}\label{4-4}
\left|\langle V_{\bar{r}}\psi,\varphi \rangle\right|
&=&\frac{\bar{r}^{2}}{\bar{r}^{n}}\left|\langle V\tilde{\psi},\tilde{\varphi}\rangle\right|\nonumber\\
&\leq&\frac{\bar{r}^{2}}{\bar{r}^{n}}\omega_{1}(s\bar{r})\|\tilde{\psi}\|_{L_{r\bar{r},s\bar{r}}^{1,2}}\|\nabla \tilde{\varphi}\|_{L^{2}(B_{r\bar{r}})}\nonumber\\
&=&\omega_{1}(s\bar{r})\|\psi\|_{L_{r,s}^{1,2}}\|\nabla \varphi\|_{L^{2}(B_{r})}\nonumber\\
&=&\tilde{\omega}_{1}(s)\|\psi\|_{L_{r,s}^{1,2}}\|\nabla \varphi\|_{L^{2}(B_{r})},
\end{eqnarray}
for $\psi\in W^{1,2}(B_{1})$ and $\varphi\in W_{0}^{1,2}(B_{1})$ with $\text{supp}\{\varphi\}\subset \overline{B}_{r}$, where $\tilde{\psi}(x)=\psi(\displaystyle\frac{x}{\bar{r}})\in W^{1,2}(B_{\bar{r}})$, $\tilde{\varphi}(x)=
\varphi(\displaystyle\frac{x}{\bar{r}})\in W_{0}^{1,2}(B_{\bar{r}})$. Similarly, we also have
\begin{eqnarray}\label{4-5}
\left|\langle g_{\bar{r}},\varphi \rangle\right|
&=&\frac{\bar{r}^{2}}{A\bar{r}^{n}}\left|\langle g,\tilde{\varphi}\rangle\right|\nonumber\\
&\leq&\frac{\bar{r}^{2}}{A\bar{r}^{n}}\frac{\omega_{2}(r\bar{r})}{(r\bar{r})^{2}}|B_{r\bar{r}}|^{\frac{n+2}{2n}}\|\nabla \tilde{\varphi}\|_{L^{2}(B_{r\bar{r}})}\nonumber\\
&=&\frac{\omega_{2}(r\bar{r})}{Ar^{2}}|B_{r}|^{\frac{n+2}{2n}}\|\nabla \varphi\|_{L^{2}(B_{r})}\nonumber\\
&=&\frac{\tilde{\omega}_{2}(r)}{Ar^{2}}|B_{r}|^{\frac{n+2}{2n}}\|\nabla \varphi\|_{L^{2}(B_{r})},
\end{eqnarray}
\begin{eqnarray}\label{4-6}
\int_{B_{r}}|\vec{f}_{\bar{r}}|^{2}|\varphi|^{2}dx&=&
\int_{B_{r}}\frac{\bar{r}^{2}}{A^{2}}|\vec{f}(\bar{r}x)|^{2}|\varphi(x)|^{2}dx\nonumber\\
&\leq&\frac{1}{A^{2}}\frac{\bar{r}^{2}}{\bar{r}^{n}}(\omega_{2}(r\bar{r}))^{2}\int_{B_{r\bar{r}}}
|\nabla\tilde{\varphi}(t)|^{2}dt\nonumber\\
&=&\frac{(\tilde{\omega}_{2}(r))^{2}}{A^{2}}\displaystyle\int_{B_{r}}|\nabla \varphi(x)|^{2}dx,
\end{eqnarray}
and by (\ref{4-1}) we know that $\tilde{\omega}_{1}(r)$ satisfies
\begin{equation}\label{omega1}
C_{0}|B_{1}|^{\frac{n+2}{2n}}
\left(\tilde{\omega}_{1}(1)+\frac{1}{1-\lambda}\int_{0}^{1}\frac{\tilde{\omega}_{1}(s)}{s}ds\right)\leq\frac{\delta_{0}}{128}.
\end{equation}

\
\

$\bf{Step~2:~Iterating~results}$\\

Now, we prove the following claim inductively: there is a series of linear functions $\{l_{k}(x)\}_{k=0}^{\infty}$ with $l_{k}(x)=a_{k}+\vec{b}_{k}\cdot x$ and a nonnegative sequence $\{T_{k}\}_{k=0}^{\infty}$ such that
\begin{equation}\label{ite}
\left(\frac{1}{|B_{\lambda^{k}}|}\int_{B_{\lambda^{k}}}|w-l_{k}(x)|^{2}dx\right)^{\frac{1}{2}}\leq T_{k},\quad \forall k\geq0,
\end{equation}
and
\begin{equation}\label{error}
|a_{k}-a_{k-1}|+\lambda^{k-1}|\vec{b}_{k}-\vec{b}_{k-1}|\leq C_{0}T_{k-1},\quad \forall k\geq1,
\end{equation}
where $a_{0}=|\vec{b}_{0}|=0$, $T_{0}=1$,
\begin{equation}\label{dtk}
T_{k}=\max\left\{\frac{1}{2}\lambda T_{k-1},~\frac{16\tilde{\omega}_{1}(\lambda^{k})|B_{1}|^{\frac{n+2}{2n}}}{\delta_{0}}(|a_{k}|+\lambda^{k}|\vec{b}_{k}|)
,~\frac{16\tilde{\omega}_{2}(\lambda^{k})|B_{1}|^{\frac{n+2}{2n}}}{A\delta_{0}}\right\}
\end{equation}
for $k=1,2,\cdots$, and $C_{0}$ is the constant in Lemma $\ref{key1}$.

Firstly by normalized assumption in the {\bf Step 1}, we know that $(\ref{ite})$  holds for $k=0$  since $a_{0}=|\vec{b}_{0}|=0$ and $T_{0}=1$, i.e.
$$\left(\displaystyle\frac{1}{|B_{1}|}\displaystyle\int_{B_{1}}w^{2}dx\right)^{\frac{1}{2}}\leq 1=T_{0}.$$
Now assume by induction that the conclusion is true for $k$. We set
$$\tilde{w}(x)=\frac{w(\lambda^{k}x)-l_{k}(\lambda^{k}x)}{T_{k}}.
$$
Since $w(x)$ satisfies (\ref{4-2}), we know that $\tilde{w}$ is a weak solution to
\begin{eqnarray*}
-\Delta \tilde{w}+\tilde{V}\tilde{w}=-\text{div}\vec{\tilde{f}}+\tilde{g}-\tilde{V}\frac{l_{k}(\lambda^{k}x)}{T_{k}},\quad x\in B_{1},
\end{eqnarray*}
where $\vec{\tilde{f}}(x)
=\displaystyle\frac{\lambda^{k}\vec{f}_{\bar{r}}(\lambda^{k}x)}{T_{k}}$, $\langle \tilde{V}\cdot,\cdot\rangle$ is bounded on $W^{1,2}(B_{1})\times W_{0}^{1,2}(B_{1})$ satisfying
$$\langle \tilde{V}\psi,\varphi\rangle=\frac{\lambda^{2k}}{\lambda^{nk}}\langle V_{\bar{r}}\tilde{\psi},\tilde{\varphi}\rangle
$$
for $\tilde{\psi}(x)=\psi(\displaystyle\frac{x}{\lambda^{k}})\in W^{1,2}(B_{\lambda^{k}}),~\tilde{\varphi}(x)=
\varphi(\displaystyle\frac{x}{\lambda^{k}})\in W_{0}^{1,2}(B_{\lambda^{k}})$, and $\tilde{g}$ is a bounded linear functional on $W_{0}^{1,2}(B_{1})$ satisfying
$$\langle \tilde{g},\varphi\rangle=\frac{\lambda^{2k}}{T_{k}\lambda^{nk}}\langle g_{\bar{r}},\tilde{\varphi}\rangle
$$
for $\tilde{\varphi}(x)=
\varphi(\displaystyle\frac{x}{\lambda^{k}})\in W_{0}^{1,2}(B_{\lambda^{k}})$.
Thus by using the inductive assumption and (\ref{4-4}), we obtain that
$$\frac{1}{|B_{1}|}\displaystyle\int_{B_{1}}\tilde{w}^{2}dx=
\frac{1}{|B_{1}|}\displaystyle\int_{B_{1}}\frac{|w(\lambda^{k}x)-l_{k}|^{2}}{T_{k}^{2}}dx
\leq1,
$$

\begin{equation}\label{tildev}
\begin{aligned}
\left|\langle\tilde{V}\psi,\varphi \rangle\right|
&=\frac{\lambda^{2k}}{\lambda^{nk}}\left|\langle V_{\bar{r}}\tilde{\psi},\tilde{\varphi}\rangle\right|\\
&\leq\frac{\lambda^{2k}}{\lambda^{nk}}\tilde{\omega}_{1}(\lambda^{k}\theta)\|\tilde{\psi}\|_{L_{\lambda^{k}\rho,\lambda^{k}\theta}^{1,2}}\|\nabla \tilde{\varphi}\|_{L^{2}(B_{\lambda^{k}\rho})}\\
&\leq\tilde{\omega}_{1}(\lambda^{k})\|\psi\|_{L_{\rho,\theta}^{1,2}}\|\nabla \varphi\|_{L^{2}(B_{\rho})}\\
&\leq\frac{\delta_{0}}{128}\|\psi\|_{L_{\rho,\theta}^{1,2}}\|\nabla \varphi\|_{L^{2}(B_{\rho})},~~~\forall \rho<\theta\leq2\rho,
\end{aligned}
\end{equation}
for any $0<\rho\leq\displaystyle\frac{1}{2}$, $\psi\in W^{1,2}(B_{1})$, $\varphi\in W_{0}^{1,2}(B_{1})$ with $\text{supp}\{\varphi\}\subset \overline{B}_{\rho}$,
$\tilde{\psi}(x)=\psi(\displaystyle\frac{x}{\lambda^{k}})\in W^{1,2}(B_{\lambda^{k}})$ and  $\tilde{\varphi}(x)=
\varphi(\displaystyle\frac{x}{\lambda^{k}})\in W_{0}^{1,2}(B_{\lambda^{k}})$. Here we have used $\tilde{\omega}_{1}(1)\leq\displaystyle\frac{\delta_{0}}{128}$ by (\ref{omega1}). Similarly, since $T_{k}\geq\displaystyle\frac{16\tilde{\omega}_{2}(\lambda^{k})|B_{1}|^{\frac{n+2}{2n}}}{A\delta_{0}}$, we have from (\ref{4-5}) and (\ref{4-6})
\begin{eqnarray*}
\left|\langle\tilde{g},\varphi \rangle\right|
&=&\frac{\lambda^{2k}}{T_{k}\lambda^{nk}}\left|\langle g_{\bar{r}},\tilde{\varphi}\rangle\right|\\
&\leq&\frac{\lambda^{2k}}{T_{k}\lambda^{nk}}\frac{\tilde{\omega}_{2}(\lambda^{k}\rho)}{A(\lambda^{k}\rho)^{2}}|B_{\lambda^{k}\rho}|^{\frac{n+2}{2n}}\|\nabla \tilde{\varphi}\|_{L^{2}(B_{\lambda^{k}\rho})}\\
&\leq&\frac{\tilde{\omega}_{2}(\lambda^{k})}{AT_{k}}|B_{1}|^{\frac{n+2}{2n}}\|\nabla \varphi\|_{L^{2}(B_{\rho})}\\
&\leq&\frac{\delta_{0}}{16}\|\nabla \varphi\|_{L^{2}(B_{\rho})},
\end{eqnarray*}
and
\begin{eqnarray*}
\int_{B_{\rho}}|\vec{\tilde{f}}|^{2}\varphi^{2}dx&\leq&\frac{\lambda^{2k}}{T_{k}^{2}}\int_{B_{\rho}}|\vec{f}_{\bar{r}}(\lambda^{k}x)|^{2}
|\varphi(x)|^{2}dx\\
&=&\frac{\lambda^{2k}}{T_{k}^{2}}\frac{1}{\lambda^{nk}}\int_{B_{\lambda^{k}\rho}}|\vec{f}_{\bar{r}}(t)|^{2}|\tilde{\varphi}(t)|^{2}dt\\
&\leq&\frac{\lambda^{2k}}{T_{k}^{2}}\frac{1}{\lambda^{nk}}\frac{(\tilde{\omega}_{2}(\lambda^{k}\rho))^{2}}{A^{2}}
\int_{B_{\lambda^{k}\rho}}|\nabla\tilde{\varphi}(t)|^{2}dt\\
&\leq&\frac{\delta_{0}^{2}}{256}\int_{B_{\rho}}|\nabla\varphi(x)|^{2}dx.
\end{eqnarray*}
Furthermore, since $\tilde{V}\displaystyle\frac{l_{k}(\lambda^{k}x)}{T_{k}}$ is also a bounded linear functional on $W_{0}^{1,2}(B_{1})$, hence by $(\ref{tildev})$ we have for any $\varphi\in W_{0}^{1,2}(B_{1})$ with $\text{supp}\{\varphi\}\subset \overline{B}_{\rho}$,
\begin{eqnarray*}
\left|\langle\tilde{V}\frac{l_{k}(\lambda^{k}x)}{T_{k}},\varphi \rangle\right|
&\leq&\frac{\tilde{\omega}_{1}(\lambda^{k})}{T_{k}}\|l_{k}(\lambda^{k}x)\|_{L_{\rho,2\rho}^{1,2}}\|\nabla \varphi\|_{L^{2}(B_{\rho})}\\
&\leq&\frac{4|B_{1}|^{\frac{n+2}{2n}}\tilde{\omega}_{1}(\lambda^{k})(|a_{k}|+\lambda^{k}|\vec{b}_{k}|)}{T_{k}}
\|\nabla \varphi\|_{L^{2}(B_{\rho})}\\
&\leq&\frac{\delta_{0}}{4}\|\nabla \varphi\|_{L^{2}(B_{\rho})},
\end{eqnarray*}
where $T_{k}
\geq\displaystyle\frac{16\tilde{\omega}_{1}(\lambda^{k})|B_{1}|^{\frac{n+2}{2n}}}{\delta_{0}}(|a_{k}|+\lambda^{k}|\vec{b}_{k}|)$ is used in the last inequality. Thus, we obtain that  for any $0<\rho\leq\displaystyle\frac{1}{2}$, $\psi\in W^{1,2}(B_{1})$, $\varphi\in W_{0}^{1,2}(B_{1})$ with $\text{supp}\{\varphi\}\subset \overline{B}_{\rho}$,
$$|\langle \tilde{V}\psi,\varphi\rangle|\leq \delta_{0}\|\psi\|_{L_{\rho,\theta}^{1,2}}\|\nabla\varphi\|_{L^{2}(B_{\rho})}, ~~~\rho<\theta\leq2\rho$$
$$|\langle \tilde{g}-\tilde{V}\frac{l_{k}(\lambda^{k}x)}{T_{k}},\varphi\rangle|\leq \delta_{0}\|\nabla\varphi\|_{L^{2}(B_{\rho})},$$
$$\displaystyle\int_{B_{\rho}}|\tilde{\vec{f}}|^{2}\varphi^{2}dx\leq\delta_{0}^{2}\displaystyle
\int_{B_{\rho}}|\nabla \varphi|^{2}dx.
$$
Now we can apply Lemma $\ref{key1}$ for $\tilde{w}$ to obtain a linear function $l(x)=a+\vec{b}\cdot x$ with $|a|+|\vec{b}|\leq C_{0}$ such that
$$\left(\frac{1}{|B_{\lambda}|}\int_{B_{\lambda}}|\tilde{w}-l(x)|^{2}dx\right)^{\frac{1}{2}}\leq\frac{1}{2}\lambda.
$$
We scale back to get
$$\left(\frac{1}{|B_{\lambda^{k+1}}|}\int_{B_{\lambda^{k+1}}}|w(x)-l_{k}(x)-T_{k}l(\frac{x}{\lambda^{k}})|^{2}dx\right)^{\frac{1}{2}}
\leq\frac{1}{2}\lambda T_{k}\leq T_{k+1}.
$$
Thus we prove the $(k+1)$-th step by letting
$$l_{k+1}(x)=l_{k}(x)+T_{k}l(\frac{x}{\lambda^{k}}).
$$
Clearly, it follows that
$$|a_{k+1}-a_{k}|+\lambda^{k}|\vec{b}_{k+1}-\vec{b}_{k}|\leq C_{0}T_{k}.
$$

{\bf{Step~3: Prove $\displaystyle\sum\limits_{k=0}^{\infty}T_{k}$ is convergent and $\{a_{k}\}_{k=0}^{\infty}$ is a Cauchy sequence.}}\\

By induction assumption $(\ref{error})$, since $a_{0}=|\vec{b}_{0}|=0$, then for any $k\geq1$,
$$|a_{k}|\leq\sum\limits_{i=0}^{k-1}|a_{i+1}-a_{i}|\leq C_{0}\sum\limits_{i=0}^{k-1}T_{i},
$$
\begin{eqnarray*}
|\vec{b}_{k}|\leq\sum\limits_{i=0}^{k-1}|\vec{b}_{i+1}-\vec{b}_{i}|\leq C_{0}\sum\limits_{i=0}^{k-1}\frac{T_{i}}{\lambda^{i}},
\end{eqnarray*}
\begin{equation}\label{bk}
\lambda^{k}|\vec{b}_{k}|\leq C_{0}\lambda^{k}\left(\sum\limits_{i=0}^{k-1}\frac{T_{i}}{\lambda^{i}}\right)
\leq C_{0}\sum\limits_{i=0}^{k-1}T_{i}.
\end{equation}
By the definition of $T_{k}$, $(\ref{dtk})$ implies that
\begin{equation}\label{Tk}
\begin{aligned}
T_{i}
&\leq\frac{1}{2}\lambda T_{i-1}+\frac{16\tilde{\omega}_{1}(\lambda^{i})|B_{1}|^{\frac{n+2}{2n}}}{\delta_{0}}(|a_{i}|+\lambda^{i}|\vec{b}_{i}|)
+\frac{16\tilde{\omega}_{2}(\lambda^{i})|B_{1}|^{\frac{n+2}{2n}}}{A\delta_{0}}\\
&\leq\lambda T_{i-1}+\frac{16\tilde{\omega}_{1}(\lambda^{i})|B_{1}|^{\frac{n+2}{2n}}}{\delta_{0}}(|a_{i}|+\lambda^{i}|\vec{b}_{i}|)
+\frac{16\tilde{\omega}_{2}(\lambda^{i})|B_{1}|^{\frac{n+2}{2n}}}{A\delta_{0}}\\
&\leq\frac{1}{2} T_{i-1}+\frac{16\tilde{\omega}_{1}(\lambda^{i})|B_{1}|^{\frac{n+2}{2n}}}{\delta_{0}}(|a_{i}|+\lambda^{i}|\vec{b}_{i}|)
+\frac{16\tilde{\omega}_{2}(\lambda^{i})|B_{1}|^{\frac{n+2}{2n}}}{A\delta_{0}},\quad i=1,2,\cdots.
\end{aligned}
\end{equation}
Hence for any fixed $k\geq1$,
\begin{eqnarray*}
\sum_{i=1}^{k}T_{i}&\leq&\frac{1}{2}\sum_{i=1}^{k}T_{i}+\frac{1}{2}T_{0}+
\sum_{i=1}^{k}\frac{16\tilde{\omega}_{1}(\lambda^{i})|B_{1}|^{\frac{n+2}{2n}}}{\delta_{0}}(|a_{i}|+\lambda^{i}|\vec{b}_{i}|)+
\sum_{i=1}^{k}\frac{16\tilde{\omega}_{2}(\lambda^{i})|B_{1}|^{\frac{n+2}{2n}}}{A\delta_{0}}\\
&\leq&\frac{1}{2}\sum_{i=1}^{k}T_{i}+\frac{1}{2}T_{0}+
\sum_{i=1}^{k}\frac{16\tilde{\omega}_{1}(\lambda^{i})|B_{1}|^{\frac{n+2}{2n}}}{\delta_{0}}\left(C_{0}\sum\limits_{l=0}^{i-1}T_{l}
+C_{0}\sum\limits_{l=0}^{i-1}T_{l}\right)+\sum_{i=1}^{k}\frac{16\tilde{\omega}_{2}(\lambda^{i})|B_{1}|^{\frac{n+2}{2n}}}{A\delta_{0}}\\
&\leq&\frac{1}{2}\sum_{i=1}^{k}T_{i}+\frac{1}{2}T_{0}+
\sum_{i=1}^{k}\frac{32C_{0}\tilde{\omega}_{1}(\lambda^{i})|B_{1}|^{\frac{n+2}{2n}}}{\delta_{0}}\sum\limits_{l=0}^{k-1}T_{l}
+\sum_{i=1}^{k}\frac{16\tilde{\omega}_{2}(\lambda^{i})|B_{1}|^{\frac{n+2}{2n}}}{A\delta_{0}}\\
&\leq&\frac{1}{2}\sum_{i=1}^{k}T_{i}+\frac{1}{2}T_{0}+
\sum_{i=0}^{k}\frac{32C_{0}\tilde{\omega}_{1}(\lambda^{i})|B_{1}|^{\frac{n+2}{2n}}}{\delta_{0}}\left(T_{0}+\sum\limits_{i=1}^{k}T_{i}\right)
+\sum_{i=0}^{k}\frac{16\tilde{\omega}_{2}(\lambda^{i})|B_{1}|^{\frac{n+2}{2n}}}{A\delta_{0}}.
\end{eqnarray*}
Due to for any $k\geq0$
$$\sum_{i=0}^{k}\tilde{\omega}_{1}(\lambda^{i})\leq
\left(\tilde{\omega}_{1}(1)+\frac{1}{1-\lambda}\int_{0}^{1}\frac{\tilde{\omega}_{1}(s)}{s}ds\right),\quad
\sum_{i=0}^{k}\tilde{\omega}_{2}(\lambda^{i})\leq
\left(\tilde{\omega}_{2}(1)+\frac{1}{1-\lambda}\int_{0}^{1}\frac{\tilde{\omega}_{2}(s)}{s}ds\right).
$$
It follows that for any $k\geq1$
\begin{eqnarray*}
\sum_{i=1}^{k}T_{i}&\leq&\frac{1}{2}\sum_{i=1}^{k}T_{i}+\frac{1}{2}T_{0}+
\frac{32C_{0}|B_{1}|^{\frac{n+2}{2n}}}{\delta_{0}}\left(\tilde{\omega}_{1}(1)
+\frac{1}{1-\lambda}\int_{0}^{1}\frac{\tilde{\omega}_{1}(s)}{s}ds\right)\left(T_{0}+\sum\limits_{i=1}^{k}T_{i}\right)\\
&&+\frac{16|B_{1}|^{\frac{n+2}{2n}}}{A\delta_{0}}\left(\tilde{\omega}_{2}(1)+\frac{1}{1-\lambda}\int_{0}^{1}\frac{\tilde{\omega}_{2}(s)}{s}ds\right).
\end{eqnarray*}
Combining with $(\ref{omega1})$ and the definition of $A$, we obtain
\begin{eqnarray*}
\sum_{i=1}^{k}T_{i}\leq\frac{1}{2}\sum_{i=1}^{k}T_{i}+\frac{1}{2}T_{0}+
\frac{1}{4}\left(T_{0}+\sum\limits_{i=1}^{k}T_{i}\right)+4.
\end{eqnarray*}
Consequently we have
$$\sum_{i=1}^{k}T_{i}\leq 3T_{0}+16.
$$
Since $T_{0}=1$, then for any $k\geq0$ it follows
$$\sum_{i=0}^{k}T_{i}\leq20.
$$
Furthermore, for any $k\geq0$,
\begin{equation}\label{lkbdd}
|a_{k}|+\lambda^{k}|\vec{b}_{k}|
\leq 40C_{0}\triangleq L.
\end{equation}
Thus we have shown $\sum\limits_{k=0}^{\infty}T_{k}$ is convergent. It is obvious that $\{a_{k}\}_{k=0}^{\infty}$ is a Cauchy sequence, which we can assume its limit is $a_{\infty}$.\\

{\bf{Step~4: Prove $w$ is continuous at 0.}}\\

First, for any $k\geq0$, we have
\begin{equation}\label{5+3}
\begin{aligned}
\left(\frac{1}{|B_{\lambda^{k}}|}\displaystyle\int_{B_{\lambda^{k}}}|w-a_{\infty}|^{2}dx\right)^{\frac{1}{2}}&\leq
\left(\frac{1}{|B_{\lambda^{k}}|}\displaystyle\int_{B_{\lambda^{k}}}|w-l_{k}(x)|^{2}dx\right)^{\frac{1}{2}}+
\sum_{i=k}^{\infty}\left(\frac{1}{|B_{\lambda^{k}}|}\displaystyle\int_{B_{\lambda^{k}}}|a_{i}-a_{i+1}|^{2}dx\right)^{\frac{1}{2}}\\
&\quad+\left(\frac{1}{|B_{\lambda^{k}}|}\displaystyle\int_{B_{\lambda^{k}}}|\vec{b}_{k}\cdot x|^{2}dx\right)^{\frac{1}{2}}\\
&\leq T_{k}+\sum_{i=k}^{\infty}|a_{i}-a_{i+1}|+\lambda^{k}|\vec{b}_{k}|\\
&\leq(1+C_{0})T_{k}+C_{0}\sum_{i=k+1}^{\infty}T_{i}+\lambda^{k}|\vec{b}_{k}|.
\end{aligned}
\end{equation}
Next we estimate $T_{k}$ and $\sum\limits_{i=k+1}^{\infty}T_{i}$. Using inequality $(\ref{Tk})$ repeatedly, we have for any $k\geq1$,
\begin{eqnarray*}
T_{k}&\leq&\lambda^{k}T_{0}+\lambda^{k}\sum_{i=1}^{k}
\frac{16\tilde{\omega}_{1}(\lambda^{i})|B_{1}|^{\frac{n+2}{2n}}}{\delta_{0}\lambda^{i}}(a_{i}+\lambda^{i}|\vec{b}_{i}|)+\lambda^{k}\sum_{i=1}^{k}
\frac{16\tilde{\omega}_{2}(\lambda^{i})|B_{1}|^{\frac{n+2}{2n}}}{A\delta_{0}\lambda^{i}}\\
&\leq&\lambda^{k}T_{0}+\frac{16\lambda^{k}L|B_{1}|^{\frac{n+2}{2n}}}{\delta_{0}\lambda(1-\lambda)}
\int_{\lambda^{k}}^{1}\frac{\tilde{\omega}_{1}(s)}{s^{2}}ds
+\frac{16\lambda^{k}|B_{1}|^{\frac{n+2}{2n}}}{A\delta_{0}\lambda(1-\lambda)}\int_{\lambda^{k}}^{1}\frac{\tilde{\omega}_{2}(s)}{s^{2}}ds
.
\end{eqnarray*}
Combining with $T_{0}=1$ we have for any $k\geq0$,
\begin{equation}\label{tk}
T_{k}\leq\lambda^{k}T_{0}+\frac{16\lambda^{k}L|B_{1}|^{\frac{n+2}{2n}}}{\delta_{0}\lambda(1-\lambda)}
\int_{\lambda^{k}}^{1}\frac{\tilde{\omega}_{1}(s)}{s^{2}}ds
+\frac{16\lambda^{k}|B_{1}|^{\frac{n+2}{2n}}}{A\delta_{0}\lambda(1-\lambda)}\int_{\lambda^{k}}^{1}\frac{\tilde{\omega}_{2}(s)}{s^{2}}ds
.
\end{equation}
Similarly, for any $k\geq0$, we have
\begin{eqnarray*}
\sum\limits_{i=k+1}^{\infty}T_{i}&\leq&\sum\limits_{i=k+1}^{\infty}\left(\frac{1}{2} T_{i-1}+\frac{16\tilde{\omega}_{1}(\lambda^{i})|B_{1}|^{\frac{n+2}{2n}}}
{\delta_{0}}(|a_{i}|+\lambda^{i}|\vec{b}_{i}|)
+\frac{16\tilde{\omega}_{2}(\lambda^{i})|B_{1}|^{\frac{n+2}{2n}}}{A\delta_{0}}\right)\\
&\leq&\frac{1}{2} T_{k}+\frac{1}{2}\sum\limits_{i=k+1}^{\infty}T_{i}+\frac{16 L|B_{1}|^{\frac{n+2}{2n}}}{\delta_{0}}\sum\limits_{i=k+1}^{\infty}
\tilde{\omega}_{1}(\lambda^{i})+\frac{16|B_{1}|^{\frac{n+2}{2n}}}{A\delta_{0}}\sum\limits_{i=k+1}^{\infty}
\tilde{\omega}_{2}(\lambda^{i})\\
&\leq&\frac{1}{2} T_{k}+\frac{1}{2}\sum\limits_{i=k+1}^{\infty}T_{i}+
\frac{16 L|B_{1}|^{\frac{n+2}{2n}}}{\delta_{0}(1-\lambda)}
\int_{0}^{\lambda^{k}}\frac{\tilde{\omega}_{1}(s)}{s}ds
+\frac{16|B_{1}|^{\frac{n+2}{2n}}}{A\delta_{0}(1-\lambda)}
\int_{0}^{\lambda^{k}}\frac{\tilde{\omega}_{2}(s)}{s}ds,
\end{eqnarray*}
Hence it follows
\begin{equation}\label{5+1}
C_{0}\sum\limits_{i=k+1}^{\infty}T_{i}\leq
C_{0} T_{k}+
\frac{32C_{0} L|B_{1}|^{\frac{n+2}{2n}}}{\delta_{0}(1-\lambda)}
\int_{0}^{\lambda^{k}}\frac{\tilde{\omega}_{1}(s)}{s}ds
+\frac{32C_{0}|B_{1}|^{\frac{n+2}{2n}}}{A\delta_{0}(1-\lambda)}
\int_{0}^{\lambda^{k}}\frac{\tilde{\omega}_{2}(s)}{s}ds.
\end{equation}
Now we estimate $\lambda^{k}|\vec{b}_{k}|$. Notice that the inequality $(\ref{Tk})$ implies that
$$\frac{T_{i}}{\lambda^{i}}\leq\frac{1}{2}\frac{T_{i-1}}{\lambda^{i-1}}+\frac{16\tilde{\omega}_{1}(\lambda^{i})|B_{1}|^{\frac{n+2}{2n}}}{\delta_{0}\lambda^{i}}(|a_{i}|+\lambda^{i}|\vec{b}_{i}|)
+\frac{16\tilde{\omega}_{2}(\lambda^{i})|B_{1}|^{\frac{n+2}{2n}}}{A\delta_{0}\lambda^{i}},\quad i=1,2,\cdots.
$$
Hence we have for any $k\geq1$
\begin{eqnarray*}
\sum_{i=1}^{k}\frac{T_{i}}{\lambda^{i}}&\leq& \frac{1}{2}\sum_{i=1}^{k}\frac{T_{i-1}}{\lambda^{i-1}}+\frac{16L|B_{1}|^{\frac{n+2}{2n}}}{\delta_{0}}\sum_{i=1}^{k}
\frac{\tilde{\omega}_{1}(\lambda^{i})}{\lambda^{i}}
+\frac{16|B_{1}|^{\frac{n+2}{2n}}}{A\delta_{0}}\sum_{i=1}^{k}\frac{\tilde{\omega}_{2}(\lambda^{i})}{\lambda^{i}}\\
&\leq&\frac{1}{2}T_{0}+\frac{1}{2}\sum_{i=1}^{k}\frac{T_{i}}{\lambda^{i}}+\frac{16L|B_{1}|^{\frac{n+2}{2n}}}{\delta_{0}\lambda(1-\lambda)}
\int_{\lambda^{k}}^{1}\frac{\tilde{\omega}_{1}(s)}{s^{2}}ds
+\frac{16|B_{1}|^{\frac{n+2}{2n}}}{A\delta_{0}\lambda(1-\lambda)}\int_{\lambda^{k}}^{1}\frac{\tilde{\omega}_{2}(s)}{s^{2}}ds.
\end{eqnarray*}
It is clear that
$$\sum_{i=1}^{k}\frac{T_{i}}{\lambda^{i}}\leq
T_{0}+\frac{32L|B_{1}|^{\frac{n+2}{2n}}}{\delta_{0}\lambda(1-\lambda)}\int_{\lambda^{k}}^{1}\frac{\tilde{\omega}_{1}(s)}{s^{2}}ds
+\frac{32|B_{1}|^{\frac{n+2}{2n}}}{A\delta_{0}\lambda(1-\lambda)}\int_{\lambda^{k}}^{1}\frac{\tilde{\omega}_{2}(s)}{s^{2}}ds.
$$
Since $T_{0}=1$,  we also have for any $k\geq0$,
$$\sum_{i=0}^{k}\frac{T_{i}}{\lambda^{i}}\leq
2T_{0}+\frac{32L|B_{1}|^{\frac{n+2}{2n}}}{\delta_{0}\lambda(1-\lambda)}\int_{\lambda^{k}}^{1}\frac{\tilde{\omega}_{1}(s)}{s^{2}}ds
+\frac{32|B_{1}|^{\frac{n+2}{2n}}}{A\delta_{0}\lambda(1-\lambda)}\int_{\lambda^{k}}^{1}\frac{\tilde{\omega}_{2}(s)}{s^{2}}ds.
$$
Therefore from $(\ref{bk})$ we have
\begin{equation}\label{5+2}\lambda^{k}|\vec{b}_{k}|\leq C_{0}\lambda^{k}\left(\sum\limits_{i=0}^{k-1}\frac{T_{i}}{\lambda^{i}}\right)
\leq C_{0}\lambda^{k}\left(2T_{0}+\frac{32L|B_{1}|^{\frac{n+2}{2n}}}{\delta_{0}\lambda(1-\lambda)}\int_{\lambda^{k}}^{1}\frac{\tilde{\omega}_{1}(s)}{s^{2}}ds
+\frac{32|B_{1}|^{\frac{n+2}{2n}}}{A\delta_{0}\lambda(1-\lambda)}\int_{\lambda^{k}}^{1}\frac{\tilde{\omega}_{2}(s)}{s^{2}}ds\right).
\end{equation}
Substituting $(\ref{tk})$, $(\ref{5+1})$ and $(\ref{5+2})$ into $(\ref{5+3})$ to get
\begin{eqnarray*}
& & \left(\frac{1}{|B_{\lambda^{k}}|}\displaystyle\int_{B_{\lambda^{k}}}|w-a_{\infty}|^{2}dx\right)^{\frac{1}{2}}\\
&\leq&
(1+C_{0})T_{k}+C_{0}\sum_{i=k+1}^{\infty}T_{i}+\lambda^{k}|\vec{b}_{k}|\\
&\leq&\left(2C_{0}+1\right)T_{k}+
\frac{32C_{0} L|B_{1}|^{\frac{n+2}{2n}}}{\delta_{0}(1-\lambda)}
\int_{0}^{\lambda^{k}}\frac{\tilde{\omega}_{1}(s)}{s}ds
+\frac{32C_{0}|B_{1}|^{\frac{n+2}{2n}}}{A\delta_{0}(1-\lambda)}
\int_{0}^{\lambda^{k}}\frac{\tilde{\omega}_{2}(s)}{s}ds+\lambda^{k}|\vec{b}_{k}|\\
&\leq&\left(4C_{0}+1\right)\lambda^{k}
+\frac{32L(C_{0}+1)|B_{1}|^{\frac{n+2}{2n}} }{\delta_{0}\lambda(1-\lambda)}
\left(\lambda^{k}\int_{\lambda^{k}}^{1}\frac{\tilde{\omega}_{1}(s)}{s^{2}}ds
+\int_{0}^{\lambda^{k}}\frac{\tilde{\omega}_{1}(s)}{s}ds\right)\\
&&+\frac{32(C_{0}+1)|B_{1}|^{\frac{n+2}{2n}} }{A\delta_{0}\lambda(1-\lambda)}
\left(\lambda^{k}\int_{\lambda^{k}}^{1}\frac{\tilde{\omega}_{2}(s)}{s^{2}}ds
+\int_{0}^{\lambda^{k}}\frac{\tilde{\omega}_{2}(s)}{s}ds\right).
\end{eqnarray*}
We denote $C_{\lambda,\delta_{0}}=\displaystyle\frac{32L(C_{0}+1)|B_{1}|^{\frac{n+2}{2n}}}
{\lambda^{\frac{n}{2}+1}(1-\lambda)\delta_{0}}$ and $\bar{C}_{\lambda,\delta_{0}}=\displaystyle\frac{2C_{\lambda,\delta_{0}}}{\lambda}$. Since for any $0<r\leq1$  there exists a $k\geq0$ such that $\lambda^{k+1}<r\leq \lambda^{k}$, we obtain
\begin{eqnarray*}
\left(\frac{1}{|B_{r}|}\displaystyle\int_{B_{r}}|w-a_{\infty}|^{2}dx\right)^{\frac{1}{2}}&\leq&
\left(\frac{1}{|B_{\lambda^{k+1}}|}\displaystyle\int_{B_{\lambda^{k}}}|w-a_{\infty}|^{2}dx\right)^{\frac{1}{2}}\\
&=&\frac{1}{\lambda^{\frac{n}{2}}}\left(\frac{1}{|B_{\lambda^{k}}|}\displaystyle\int_{B_{\lambda^{k}}}|w-a_{\infty}|^{2}dx\right)^{\frac{1}{2}}\\
&\leq&C_{\lambda,\delta_{0}}\lambda^{k}+C_{\lambda,\delta_{0}}
\left(\lambda^{k}\int_{\lambda^{k}}^{1}\frac{\tilde{\omega}_{1}(s)}{s^{2}}ds
+\int_{0}^{\lambda^{k}}\frac{\tilde{\omega}_{1}(s)}{s}ds\right)\\
&&+\frac{C_{\lambda,\delta_{0}}}{A}
\left(\lambda^{k}\int_{\lambda^{k}}^{1}\frac{\tilde{\omega}_{2}(s)}{s^{2}}ds
+\int_{0}^{\lambda^{k}}\frac{\tilde{\omega}_{2}(s)}{s}ds\right)\\
&\leq&C_{\lambda,\delta_{0}}\frac{r}{\lambda}+C_{\lambda,\delta_{0}}
\left(\frac{r}{\lambda}
\int_{r}^{1}\frac{\tilde{\omega}_{1}(s)}{s^{2}}ds
+\frac{2}{\lambda}\int_{0}^{\frac{r}{\lambda}}\frac{\tilde{\omega}_{1}(\lambda s)}{s}ds\right)\\
&&+\frac{C_{\lambda,\delta_{0}}}{A}
\left(\frac{r}{\lambda}\int_{r}^{1}\frac{\tilde{\omega}_{2}(s)}{s^{2}}ds
+\frac{2}{\lambda}\int_{0}^{\frac{r}{\lambda}}\frac{\tilde{\omega}_{2}(\lambda s)}{s}ds\right)\\
&\leq&\bar{C}_{\lambda,\delta_{0}}r+\bar{C}_{\lambda,\delta_{0}}
\left(r
\int_{r}^{1}\frac{\tilde{\omega}_{1}(s)}{s^{2}}ds
+\int_{0}^{r}\frac{\tilde{\omega}_{1}(s)}{s}ds\right)\\
&&+\frac{\bar{C}_{\lambda,\delta_{0}}}{A}
\left(r\int_{r}^{1}\frac{\tilde{\omega}_{2}(s)}{s^{2}}ds
+\int_{0}^{r}\frac{\tilde{\omega}_{2}(s)}{s}ds\right).
\end{eqnarray*}
where the following inequalities are used above, which can be obtained by Remark $\ref{remark1.3}$,
$$\frac{\tilde{\omega}_{1}(s)}{s}\leq\frac{2\tilde{\omega}_{1}(\lambda s)}{\lambda s},\quad
\frac{\tilde{\omega}_{2}(s)}{s}\leq\frac{2\tilde{\omega}_{2}(\lambda s)}{\lambda s}.
$$
By L'Hospital principle, we have
$$\lim\limits_{r\rightarrow0}r\int_{r}^{1}\frac{\tilde{\omega}_{1}(s)}{s^{2}}ds
=\lim\limits_{r\rightarrow0}r\int_{r}^{1}\frac{\tilde{\omega}_{2}(s)}{s^{2}}ds=0.
$$
Hence
$$\lim\limits_{r\rightarrow0}\bar{C}_{\lambda,\delta_{0}}
\left(r+r
\int_{r}^{1}\frac{\tilde{\omega}_{1}(s)}{s^{2}}ds
+\int_{0}^{r}\frac{\tilde{\omega}_{1}(s)}{s}ds\right)=0,
$$
$$\lim\limits_{r\rightarrow0}\frac{\bar{C}_{\lambda,\delta_{0}}}{A}
\left(r\int_{r}^{1}\frac{\tilde{\omega}_{2}(s)}{s^{2}}ds
+\int_{0}^{r}\frac{\tilde{\omega}_{2}(s)}{s}ds\right)=0.
$$
This implies that $w$ is continuous at 0 in the $L^{2}$ sense.\\

{\bf{Step~5: Scaling back to $u$.}}\\

We notice that $w(x)=\displaystyle\frac{u(\bar{r}x)}{A}$ and $\tilde{\omega}_{i}(r)=\omega_i(\bar{r}r)$ for $i=1, 2$. By setting $K=Aa_{\infty}$, we have for any $0<r\leq 1$,
\begin{eqnarray*}
\left(\frac{1}{|B_{\bar{r}r}|}\displaystyle\int_{B_{\bar{r}r}}|u-K|^{2}dx\right)^{\frac{1}{2}}
&\leq&\bar{C}_{\lambda,\delta_{0}}A
\left(r+r\bar{r}
\int_{r\bar{r}}^{\bar{r}}\frac{\omega_{1}(s)}{s^{2}}ds
+\int_{0}^{r\bar{r}}\frac{\omega_{1}(s)}{s}ds\right)\\
&&+\bar{C}_{\lambda,\delta_{0}}
\left(r\bar{r}
\int_{r\bar{r}}^{\bar{r}}\frac{\omega_{2}(s)}{s^{2}}ds
+\int_{0}^{r\bar{r}}\frac{\omega_{2}(s)}{s}ds\right).
\end{eqnarray*}
Hence for any $0<r\leq\bar{r}\leq R$ we obtain
\begin{eqnarray*}
\left(\frac{1}{|B_{r}|}\displaystyle\int_{B_{r}}|u-K|^{2}dx\right)^{\frac{1}{2}}&\leq&
\bar{C}_{\lambda,\delta_{0}}A
\left(\frac{r}{\bar{r}}+r
\int_{r}^{\bar{r}}\frac{\omega_{1}(s)}{s^{2}}ds
+\int_{0}^{r}\frac{\omega_{1}(s)}{s}ds\right)\\
&&+\bar{C}_{\lambda,\delta_{0}}
\left(r
\int_{r}^{\bar{r}}\frac{\omega_{2}(s)}{s^{2}}ds
+\int_{0}^{r}\frac{\omega_{2}(s)}{s}ds\right).
\end{eqnarray*}
Due to
\begin{eqnarray*}
|K|\leq A|a_{\infty}|\leq LA\leq \bar{C}_{\lambda,\delta_{0}}A,
\end{eqnarray*}
it follows that for $\bar{r}<r\leq R$,
\begin{eqnarray*}
\left(\frac{1}{|B_{r}|}\displaystyle\int_{B_{r}}|u-K|^{2}dx\right)^{\frac{1}{2}}
\leq\left(\frac{1}{|B_{\bar{r}}|}\displaystyle\int_{B_{R}}|u|^{2}dx\right)^{\frac{1}{2}}+K\leq(L+1)A\leq \bar{C}_{\lambda,\delta_{0}}A.
\end{eqnarray*}
Thus we complete the proof of Theorem \ref{linfty}.
\
\

\begin{rem}
To prove Corollary $\ref{holder}$, we only need to set $\omega_{1}(r)=N_{1}r^{\alpha_{1}}$, $\omega_{2}(r)=N_{2}r^{\alpha_{2}}$ in Theorem $\ref{linfty}$. Similarly, by taking $\omega_{1}(r)=N_{1}r$, $\omega_{2}(r)=N_{2}r$, then Corollary $\ref{holder}$ follows.
\end{rem}

\section{Continuity of solutions under Kato conditions}

In this section, we will prove the local boundedness and continuity of solution for $(\ref{1})$ mainly under the assumptions that $V\in K_{\eta_{V}}(B_{1})$, $g\in K_{\eta_{g}}(B_{1})$ and $\vec{f}\in K^{1}_{\eta_{f}}(B_{1})\cap L^{2}(B_{1})^{n}$. To begin with, we introduce the definitions and some properties of Kato class and $K^{1}$ class.

\begin{lm}\label{Vex}
Assume $V\in K_{\eta}(B_{1})$. If $\tilde{V}$ is the zero extension of $V$ in $\mathbb{R}^{n}$, then $\tilde{V}\in K_{\eta}(\mathbb{R}^{n})$ and
\begin{equation}\label{extenV}
\|\tilde{V}\|_{K(\mathbb{R}^{n})}\leq \|V\|_{K(B_{1})}.
\end{equation}
\end{lm}
\begin{proof}
In fact,
$$\|\tilde{V}\|_{K(\mathbb{R}^{n})}=\sup\limits_{x\in\mathbb{R}^{n}}\int_{\mathbb{R}^{n}}\frac{|\tilde{V}(y)|}{|x-y|^{n-2}}dy=\sup\limits_{x\in \mathbb{R}^{n}}\int_{B_{1}}\frac{|V(y)|}{|x-y|^{n-2}}dy.
$$
We only need to discuss the case $x\notin B_{1}$. For any fixed $x\notin B_{1}$, any $y\in B_{1}$, the following inequality holds,
$$\frac{1}{|x-y|^{n-2}}\leq\frac{1}{|\frac{x}{|x|}-y|^{n-2}}.
$$
Hence
$$\int_{B_{1}}\frac{|V(y)|}{|x-y|^{n-2}}dy\leq\int_{B_{1}}\frac{|V(y)|}{|\frac{x}{|x|}-y|^{n-2}}dy.
$$
Next we take the supreme over all $x\notin B_{1}$, then
\begin{eqnarray*}
\sup\limits_{x\notin B_{1}}\int_{B_{1}}\frac{|V(y)|}{|x-y|^{n-2}}dy&\leq&\sup\limits_{x\notin B_{1}}\int_{B_{1}}\frac{|V(y)|}{|\frac{x}{|x|}-y|^{n-2}}dy\\
&=&\sup\limits_{x\in\partial B_{1}}\int_{B_{1}}\frac{|V(y)|}{|x-y|^{n-2}}dy\\
&\leq&\sup\limits_{x\in B_{1}}\int_{B_{1}}\frac{|V(y)|}{|x-y|^{n-2}}dy,
\end{eqnarray*}
where Remark $\ref{rem2.9}$ is used in the last inequality. Combining with the case $x\in B_{1}$, we prove the inequality $(\ref{extenV})$. Similarly, we can calculate
$$\sup_{x\in\mathbb{R}^{n}}\int_{B_{r}(x)}\frac{|\tilde{V}(y)|}{|x-y|^{n-2}}dy\leq\eta(r),
$$
this implies that $\tilde{V}\in K_{\eta}(\mathbb{R}^{n})$.
\end{proof}
Based on this lemma, we have the following corollary of Lemma $\ref{lm5.4}$.
\begin{crl}\label{cor5.5}
Especially $\Omega=B_{1}$, $u\in H_{0}^{1}(B_{1})$, by extending $u$ and $V$ to zero outside $B_{1}$, then by Hardy inequality and Lemma $\ref{Vex}$, it is easy to get
\begin{equation}\label{37}
\int_{B_{1}}|V(x)|u^{2}(x)dx\leq C\left(\sup_{x\in B_{1}}\int_{B_{1}}\frac{|V(y)|}{|x-y|^{n-2}}dy\right)
\|\nabla u\|_{L^{2}(B_{1})}^{2}.
\end{equation}
\end{crl}
In the following, we denote $\zeta$ be the standard mollifier, i.e. $\zeta\in C^{\infty}(\mathbb{R}^{n})$ defined by
\begin{eqnarray*}
\zeta(x)=\begin{cases}
C\exp(\frac{1}{|x|^{2}-1}), \quad &|x|<1,\\
0,\quad &|x|\geq1,\\
\end{cases}
\end{eqnarray*}
the constant $C$ is selected so that $\displaystyle\int_{\mathbb{R}^{n}}\zeta dx=1$. Then for each $\epsilon>0$, set
$$\zeta_{\epsilon}(x)=\frac{1}{\epsilon^{n}}\zeta\left(\frac{x}{\epsilon}\right),
$$
it follows that $\zeta_{\epsilon}\in C^{\infty}(\mathbb{R}^{n})$ and $\displaystyle\int_{\mathbb{R}^{n}}\zeta_{\epsilon} dx=1$.
\begin{lm}
Assume that $V\in K_{\eta}(B_{1})$ and $\tilde{V}$ is the zero extension of $V$ in $\mathbb{R}^{n}$, denote  $\tilde{V}$'s mollification by
$$\tilde{V}_{\epsilon}:=\zeta_{\epsilon}\ast \tilde{V}=\displaystyle\int_{\mathbb{R}^{n}}\zeta_{\epsilon}(x-y)\tilde{V}(y)dy,\quad x\in\mathbb{R}^{n}.
$$
Then $\tilde{V}_{\epsilon}\in K_{\eta}(\mathbb{R}^{n})$ and
\begin{equation}\label{Veps}
\|\tilde{V}_{\epsilon}\|_{K(\mathbb{R}^{n})}\leq\|\tilde{V}\|_{K(\mathbb{R}^{n})}\leq\|V\|_{K(B_{1})}.
\end{equation}
\end{lm}
\begin{proof}
This lemma can be shown by  some straight calculations. In fact, for any $x$ fixed,
\begin{eqnarray*}
\int_{B_{r}(x)}\frac{|\tilde{V}_{\epsilon}(y)|}{|x-y|^{n-2}}dy&=&\int_{B_{r}(x)}\frac{1}{|x-y|^{n-2}}\left|\int_{\mathbb{R}^{n}}
\zeta_{\epsilon}(z)\tilde{V}(y-z)dz\right|dy\\
&\leq&\int_{\mathbb{R}^{n}}\zeta_{\epsilon}(z)\left(\int_{B_{r}(x)}\frac{|\tilde{V}(y-z)|}{|x-y|^{n-2}}dy\right)dz\\
&\leq&\int_{\mathbb{R}^{n}}\zeta_{\epsilon}(z)\left(\int_{|z+t-x|\leq r}\frac{|\tilde{V}(t)|}{|x-z-t|^{n-2}}dt\right)dz\\
&\leq&\int_{B_{\epsilon}(x)}\zeta_{\epsilon}(x-y)\left(\int_{|t-y|\leq r}\frac{|\tilde{V}(t)|}{|y-t|^{n-2}}dt\right)dy.
\end{eqnarray*}
Hence
\begin{eqnarray*}
\sup\limits_{x\in\mathbb{R}^{n}}\int_{B_{r}(x)}\frac{|\tilde{V}_{\epsilon}(y)|}{|x-y|^{n-2}}dy\leq
\sup\limits_{x\in\mathbb{R}^{n}}\int_{B_{\epsilon}(x)}\zeta_{\epsilon}(x-y)\sup\limits_{y\in\mathbb{R}^{n}}\left(\int_{|t-y|\leq r}\frac{|\tilde{V}(t)|}{|y-t|^{n-2}}dt\right)dy\leq\eta(r).
\end{eqnarray*}
This leads to $\tilde{V}_{\epsilon}\in K_{\eta}(\mathbb{R}^{n})$ and $(\ref{Veps})$ also holds.
\end{proof}

For more properties of mollification, we refer the readers to Appendix C.5 in \cite{E2010}.
\begin{crl}
Assume $\vec{f}\in K^{1}_{\eta}(B_{1})$. If $\vec{\tilde{f}}=(\tilde{f}_{1},\tilde{f}_{2},\cdots,\tilde{f}_{n})$ is the zero extension of $\vec{f}$ in $\mathbb{R}^{n}$, denote $\vec{\tilde{f}}_{\epsilon}$ is the mollification of $\vec{\tilde{f}}$, i.e. for $1\leq i\leq n$,
$$\tilde{f}_{\epsilon,i}:=\zeta_{\epsilon}\ast \tilde{f}_{i}=\displaystyle\int_{\mathbb{R}^{n}}\zeta_{\epsilon}(x-y)\tilde{f}_{i}(y)dy,\quad x\in\mathbb{R}^{n}.
$$
Then $\vec{\tilde{f}}\in K^{1}_{\eta}(\mathbb{R}^{n})$, $\vec{\tilde{f}}_{\epsilon}\in K^{1}_{\eta}(\mathbb{R}^{n})$ and
\begin{equation}\label{extenvecf}
\|\vec{\tilde{f}}_{\epsilon}\|_{K^{1}(\mathbb{R}^{n})}\leq\|\vec{\tilde{f}}\|_{K^{1}(\mathbb{R}^{n})}\leq
\|\vec{f}\|_{K^{1}(B_{1})}.
\end{equation}
\end{crl}

%
Next, we begin to prove the continuity of solutions, the main tools are the following existence results.
\begin{lm}\label{lm5.8}
For any $g\in K_{\eta_{g}}(B_{1})$, $\vec{f}\in K^{1}_{\eta_{f}}(B_{1})\cap L^{2}(B_{1})^{n}$, there exists a unique solution in $W_{0}^{1,2}(B_{1})\cap C(\overline{B}_{1})$ of
\begin{equation}\label{keq1}
\left\{
\begin{array}{rcll}
-\Delta u&=&-\emph{div}\vec{f}+g,\qquad&\text{in}~~ B_{1},\\
u&=&0,\qquad&\text{on}~~\partial B_{1}.
\end{array}
\right.
\end{equation}
Furthermore, the solution can be represented as
\begin{eqnarray*}
u(x)=\left\{
\begin{array}{rcll}
&\displaystyle\int_{B_{1}}\nabla_{y}G(x,y)\vec{f}(y)+G(x,y)g(y)dy,\qquad&x\in B_{1},\\
&0,\qquad&x\in\partial B_{1},
\end{array}
\right.
\end{eqnarray*}
where $G(x,y)$ is the Green's function for the operator $-\Delta$ with zero Dirichlet boundary value on $\partial B_{1}$, and the following estimate holds:
\begin{equation}\label{uest}
\|u\|_{L^{\infty}(B_{1})}+\|\nabla u\|_{L^{2}(B_{1})}\leq
C\left(\|\vec{f}\|_{L^{2}(B_{1})}+\|g\|_{L^{1}(B_{1})}+\|\vec{f}\|_{K^{1}(B_{1})}+\|g\|_{K(B_{1})}\right).
\end{equation}
\end{lm}
\begin{proof}
We denote $\vec{\tilde{f}},\tilde{g}$ are the zero extension of $\vec{f}$ and $g$ in $\mathbb{R}^{n}$ respectively, $\vec{\tilde{f}}_{\epsilon},\tilde{g}_{\epsilon}$ are the mollification of $\vec{\tilde{f}},\tilde{g}$. By the properties of mollification, $\vec{\tilde{f}}_{\epsilon}\in C^{\infty}(\mathbb{R}^{n})^{n}$, $\tilde{g}_{\epsilon}\in C^{\infty}(\mathbb{R}^{n})$ and $\vec{\tilde{f}}_{\epsilon}\rightarrow \vec{\tilde{f}}$ in $L^{2}_{\text{loc}}(\mathbb{R}^{n})^{n}$, $\tilde{g}_{\epsilon}\rightarrow \tilde{g}$ in $L^{1}_{\text{loc}}(\mathbb{R}^{n})$ as $\epsilon\rightarrow0$. Now we consider $\{u_{\epsilon}\}$ are the solutions of the following Dirichlet problems:
\begin{equation}\label{uep}
\left\{
\begin{array}{rcll}
-\Delta u_{\epsilon}&=&-\text{div}\vec{\tilde{f}}_{\epsilon}+\tilde{g}_{\epsilon},\qquad&\text{in}~~ B_{1},\\
u_{\epsilon}&=&0,\qquad&\text{on}~~\partial B_{1}.
\end{array}
\right.
\end{equation}
By the classical regularity result of elliptic equations, it follows that $u_{\epsilon}\in C^{\infty}(\overline{B}_{1})$ for any $\epsilon>0$. Moreover, $u_{\epsilon}$ can be represented as
\begin{eqnarray*}
u_{\epsilon}(x)&=&\int_{B_{1}}G(x,y)(-\text{div}\vec{\tilde{f}}_{\epsilon}(y)+\tilde{g}_{\epsilon}(y))dy\\
&=&\int_{B_{1}}\nabla_{y}G(x,y)\vec{\tilde{f}}_{\epsilon}(y)+G(x,y)\tilde{g}_{\epsilon}(y)dy.
\end{eqnarray*}
Since $|\nabla_{y}G(x,y)|\leq \displaystyle\frac{C}{|x-y|^{n-1}}$, $|G(x,y)|\leq \displaystyle\frac{C}{|x-y|^{n-2}}$, then we have for any $\epsilon>0$,
\begin{equation}\label{uepli}
\begin{aligned}
\|u_{\epsilon}\|_{L^{\infty}(B_{1})}&\leq C\left(\sup\limits_{x\in\mathbb{R}^{n}}\int_{\mathbb{R}^{n}}\frac{|\vec{\tilde{f}}_{\epsilon}(y)|}{|x-y|^{n-1}}dy
+\sup\limits_{x\in\mathbb{R}^{n}}\int_{\mathbb{R}^{n}}\frac{|\tilde{g}_{\epsilon}(y)|}{|x-y|^{n-2}}dy\right)\\
&\leq C\left(\|\vec{f}\|_{K^{-1}(B_{1})}+\|g\|_{K(B_{1})}\right).
\end{aligned}
\end{equation}
We set
\begin{eqnarray*}
\bar{u}(x)=\left\{
\begin{array}{rcll}
&\displaystyle\int_{B_{1}}\nabla_{y}G(x,y)\vec{f}(y)+G(x,y)g(y)dy,\qquad&x\in B_{1},\\
&0,\qquad&x\in\partial B_{1}.
\end{array}
\right.
\end{eqnarray*}
Since $g\in K_{\eta_{g}}(B_{1})$, $\vec{f}\in K^{1}_{\eta_{f}}(B_{1})\cap L^{2}(B_{1})^{n}$, by the definition of Kato class and $K^{1}$ class, we know that $$\lim\limits_{r\rightarrow0}(\eta_{f}(r)+\eta_{g}(r))=0.$$
It follows that, for any $\varepsilon>0$, there exists small $r_{0}$, such that for $r\leq r_{0}$,
\begin{equation}\label{etafg}
\eta_{f}(r)+\eta_{g}(r)\leq\varepsilon.
\end{equation}
Then by a straightforward calculation, for any $x\in B_{1}$,
\begin{eqnarray*}
u_{\epsilon}(x)-\bar{u}(x)&=&\int_{B_{1}}\nabla_{y}G(x,y)(\vec{\tilde{f}}_{\epsilon}(y)-\vec{f}(y))+G(x,y)(\tilde{g}_{\epsilon}(y)-g(y))dy\\
&=&\int_{B_{r_{0}}(x)}+\int_{B_{1}\setminus B_{r_{0}}(x)}\nabla_{y}G(x,y)(\vec{\tilde{f}}_{\epsilon}(y)-\vec{f}(y))+G(x,y)(\tilde{g}_{\epsilon}(y)-g(y))dy\\
&\triangleq& I_{1}+I_{2}.
\end{eqnarray*}
By $(\ref{etafg})$, it follows that for any $x\in B_{1}$,
$$|I_{1}|\leq C\varepsilon.
$$
For $I_{2}$, we have
$$|I_{2}|\leq\frac{C}{r_{0}^{n-1}}\left(\|\vec{\tilde{f}}_{\epsilon}-\vec{f}\|_{L^{1}(B_{1})}+\|\tilde{g}_{\epsilon}-g\|_{L^{1}(B_{1})}\right).
$$
Since $\vec{\tilde{f}}_{\epsilon}\rightarrow \vec{\tilde{f}}$ in $L^{2}_{\text{loc}}(\mathbb{R}^{n})^{n}$, $\tilde{g}_{\epsilon}\rightarrow \tilde{g}$ in $L^{1}_{\text{loc}}(\mathbb{R}^{n})$ as $\epsilon\rightarrow0$, then there exists small $\epsilon_{0}>0$, such that for $\epsilon\leq\epsilon_{0}$,
$$\|\vec{\tilde{f}}_{\epsilon}-\vec{f}\|_{L^{1}(B_{1})}\leq\varepsilon r_{0}^{n-1},\quad
\|\tilde{g}_{\epsilon}-g\|_{L^{1}(B_{1})}\leq\varepsilon r_{0}^{n-1}.
$$
Substituting the above inequalities into $|I_{2}|$, it follows that for any $x\in B_{1}$,
$$|I_{2}|\leq 2C\varepsilon.
$$
Combining the estimates of both $I_{1}$ and $I_{2}$ with $u_{\epsilon}=\bar{u}=0$ on $\partial B_{1}$, then we have for any $x\in \overline{B}_{1}$,
$$|u_{\epsilon}(x)-\bar{u}(x)|\leq|I_{1}|+|I_{2}|\leq C\varepsilon.
$$
This means $u_{\epsilon}(x)$ converge to $\bar{u}(x)$ uniformly in $\overline{B}_{1}$. Since $u_{\epsilon}\in C^{\infty}(\overline{B}_{1})$, it follows that $\overline{u}\in C({\overline{B}_{1}})$.

Next, we consider the convergence of $u_{\epsilon}$. Since $u_{\epsilon}\in C^{\infty}(\overline{B}_{1})$ is the solution of $(\ref{uep})$, then for sufficiently small $\epsilon>0$, using $(\ref{uepli})$,
\begin{eqnarray*}
\int_{B_{1}}|\nabla u_{\epsilon}|^{2}dx&=& \int_{B_{1}}\nabla u_{\epsilon}\cdot \vec{\tilde{f}}_{\epsilon}+u_{\epsilon}\tilde{g}_{\epsilon}dx\\
&\leq&2\int_{B_{1}}|\vec{\tilde{f}}_{\epsilon}|^{2}dx+\frac{1}{2}\int_{B_{1}}|\nabla u_{\epsilon}|^{2}dx+\|\tilde{g}_{\epsilon}\|_{L^{1}(B_{1})}\|u_{\epsilon}\|_{L^{\infty}(B_{1})}\\
&\leq&2\int_{B_{1}}|\vec{f}|^{2}dx+\frac{1}{2}\int_{B_{1}}|\nabla u_{\epsilon}|^{2}dx+C\|g\|_{L^{1}(B_{1})}\left(\|\vec{f}\|_{K^{1}(B_{1})}+\|g\|_{K(B_{1})}\right).
\end{eqnarray*}
Hence we have for sufficiently small $\epsilon>0$,
\begin{equation}\label{uepen}
\begin{aligned}
\int_{B_{1}}|\nabla u_{\epsilon}|^{2}dx
&\leq4\int_{B_{1}}|\vec{f}|^{2}dx+2C\|g\|_{L^{1}(B_{1})}\left(\|\vec{f}\|_{K^{1}(B_{1})}+\|g\|_{K(B_{1})}\right)\\
&\leq4\|\vec{f}\|_{L^{2}(B_{1})}^{2}+\|g\|^{2}_{L^{1}(B_{1})}+C\left(\|\vec{f}\|_{K^{1}(B_{1})}+\|g\|_{K(B_{1})}\right)^{2}.
\end{aligned}
\end{equation}
It follows that $\{u_{\epsilon}\}$ has a subsequence, we still denote it by $\{u_{\epsilon}\}$, such that
$$u_{\epsilon}\rightharpoonup u\quad\text{in}~W_0^{1,2}(B_{1}),~~\text{as}~~\epsilon\rightarrow0,
$$
$$u_{\epsilon}\rightarrow u\quad\text{in}~L^{2}(B_{1}),~~\text{as}~~\epsilon\rightarrow0.
$$
We claim $u\in W_0^{1,2}(B_{1})$ is a solution of
$$
-\Delta u=-\text{div}\vec{f}+g,\qquad\text{in}~~ B_{1}.
$$
In fact, for any fixed $\varphi\in C_{0}^{\infty}(B_{1})$,
\begin{equation}\label{uepd}\int_{B_{1}}\nabla u_{\epsilon}\cdot\nabla\varphi dx=\int_{B_{1}}\vec{\tilde{f}}_{\epsilon}\cdot\nabla\varphi+\tilde{g}_{\epsilon}\cdot\varphi dx.
\end{equation}
Using H\"{o}lder inequality,
$$\left|\int_{B_{1}}(\vec{\tilde{f}}_{\epsilon}-\vec{f})\cdot\nabla\varphi dx\right|\leq C\|\vec{\tilde{f}}_{\epsilon}-\vec{f}\|_{L^{2}(B_{1})}\rightarrow0,~~\text{as}~~\epsilon\rightarrow0,
$$
$$\left|\int_{B_{1}}(\tilde{g}_{\epsilon}-g)\varphi dx\right|\leq C\|\tilde{g}_{\epsilon}-g\|_{L^{1}(B_{1})}\rightarrow0,~~\text{as}~~\epsilon\rightarrow0.
$$
Then let $\epsilon\rightarrow0$ in $(\ref{uepd})$, we have
$$\int_{B_{1}}\nabla u\cdot\nabla\varphi dx=\int_{B_{1}}\vec{f}\cdot\nabla\varphi+g\cdot\varphi dx,
$$
which implies the claim. Putting $(\ref{uepli})$ and $(\ref{uepen})$ together, we obtain
\begin{equation}
\|u_{\epsilon}\|_{L^{\infty}(B_{1})}+\|\nabla u_{\epsilon}\|_{L^{2}(B_{1})}\leq
C\left(\|\vec{f}\|_{L^{2}(B_{1})}+\|g\|_{L^{1}(B_{1})}+\|\vec{f}\|_{K^{1}(B_{1})}+\|g\|_{K(B_{1})}\right).
\end{equation}
Since $u_{\epsilon}\rightarrow u$ in $L^{2}(B_{1})$, as $\epsilon\rightarrow0$, this implies that $u_{\epsilon}\rightarrow u$ a.e. in $B_{1}$. On the other hand, since $u_{\epsilon}(x)$ converge to $\bar{u}(x)$ uniformly in $\overline{B}_{1}$, as $\epsilon\rightarrow0$, thus we have that $\bar{u}=u$ in $B_{1}$. This implies that  $\bar{u}\in W_{0}^{1,2}(B_{1})\cap C(\overline{B}_{1})$ is the solution of $(\ref{keq1})$ and satisfies (\ref{uest}).
\end{proof}

\begin{thm}\label{thm5.9}
For any $g\in K_{\eta_{g}}(B_{1})$, $\vec{f}\in K^{1}_{\eta_{f}}(B_{1})\cap L^{2}(B_{1})^{n}$, $V\in K_{\eta_{V}}(B_{1})$ with
$$\|V\|_{K(B_{1})}=\sup\limits_{x\in B_{1}}\int_{B_{1}}\frac{|V(y)|}{|x-y|^{n-2}}dy\leq\delta
$$
for some $\delta$ sufficiently small, then there exists a $W_{0}^{1,2}(B_{1})\cap C(\overline{B}_{1})$ solution of
\begin{equation}\label{keq2}
\left\{
\begin{array}{rcll}
-\Delta u+Vu&=&-\emph{div}\vec{f}+g,\qquad&\text{in}~~ B_{1},\\
u&=&0,\qquad&\text{on}~~\partial B_{1},
\end{array}
\right.
\end{equation}
with the estimate
\begin{equation}\label{dves}
\|u\|_{L^{\infty}(B_{1})}+\|\nabla u\|_{L^{2}(B_{1})}\leq
C\left(\|\vec{f}\|_{L^{2}(B_{1})}+\|\vec{f}\|_{K^{1}(B_{1})}+\|g\|_{K(B_{1})}\right),
\end{equation}
where $C$ is a constant depending on $n$ and $\|V\|_{K(B_{1})}$.
\end{thm}
\begin{proof}
We prove this theorem by the fixed point theorem. Consider $w$ is a solution of
\begin{equation}\label{w}
\left\{
\begin{array}{rcll}
-\Delta w&=&-Vu-\text{div}\vec{f}+g,\qquad&\text{in}~~ B_{1},\\
w&=&0,\qquad&\text{on}~~\partial B_{1}.
\end{array}
\right.
\end{equation}
It is easy to see if $u\in C(\overline{B}_{1})$ with $\|u\|_{L^{\infty}(B_{1})}\leq A$ for some $A>0$, then $Vu$ still belongs to Kato class, i.e. $V\in K_{A\eta_{V}}(B_{1})$. Then by Lemma $\ref{lm5.8}$, it follows that there exists a unique solution $w\in W_{0}^{1,2}(B_{1})\cap C(\overline{B}_{1})$ of $(\ref{w})$. We set a mapping from $W_{0}^{1,2}(B_{1})\cap C(\overline{B}_{1})$ to itself:
$$T:~u\rightarrow w.
$$
Using the inequality $(\ref{uest})$ and Corollary \ref{cor5.5}, for $u_{1},u_{2}\in C(\overline{B}_{1})$, there exists unique $w_{1},w_{2}$ and
\begin{eqnarray*}
& & \|w_{1}-w_{2}\|_{L^{\infty}(B_{1})}+\|\nabla w_{1}-\nabla w_{2}\|_{L^{2}(B_{1})}\\
&=&\|Tu_{1}-Tu_{2}\|_{L^{\infty}(B_{1})}+\|\nabla Tu_{1}-\nabla Tu_{2}\|_{L^{2}(B_{1})}\\
&\leq&C\left(\|V(u_{1}-u_{2})\|_{L^{1}(B_{1})}+\|V(u_{1}-u_{2})\|_{K(B_{1})}\right)\\
&\leq& C\|V\|_{K(B_{1})}\|\nabla u_{1}-\nabla u_{2}\|_{L^{2}(B_{1})}+C\delta\|u_{1}-u_{2}\|_{L^{\infty}(B_{1})}\\
&\leq& C^{'}\delta\left(\|\nabla u_{1}-\nabla u_{2}\|_{L^{2}(B_{1})}+\|u_{1}-u_{2}\|_{L^{\infty}(B_{1})}\right).
\end{eqnarray*}
This means $T$ is a contraction mapping on $W_{0}^{1,2}(B_{1})\cap C(\overline{B}_{1})$ since $\delta$ is small enough. By the fixed point theorem, there exists a unique fixed point $u\in W_{0}^{1,2}(B_{1})\cap C(\overline{B}_{1})$, which is a solution to the Dirichlet problem $(\ref{keq2})$. Furthermore by using the inequality $(\ref{uest})$ again, we have,
\begin{eqnarray*}
\|u\|_{L^{\infty}(B_{1})}+\|\nabla u\|_{L^{2}(B_{1})}&\leq&
C\left(\|\vec{f}\|_{L^{2}(B_{1})}+\|g\|_{L^{1}(B_{1})}+\|\vec{f}\|_{K^{1}(B_{1})}+\|g\|_{K(B_{1})}\right)\\
&&+
C\left(\|Vu\|_{L^{1}(B_{1})}+\|Vu\|_{K(B_{1})}\right)\\
&\leq&
C\left(\|\vec{f}\|_{L^{2}(B_{1})}+\|g\|_{L^{1}(B_{1})}+\|\vec{f}\|_{K^{1}(B_{1})}+\|g\|_{K(B_{1})}\right)\\
&&+
C\delta\left(\|u\|_{L^{\infty}(B_{1})}+\|\nabla u\|_{L^{2}(B_{1})}\right),
\end{eqnarray*}
it follows that
\begin{eqnarray*}\|u\|_{L^{\infty}(B_{1})}+\|\nabla u\|_{L^{2}(B_{1})}&\leq&
C\left(\|\vec{f}\|_{L^{2}(B_{1})}+\|g\|_{L^{1}(B_{1})}+\|\vec{f}\|_{K^{1}(B_{1})}+\|g\|_{K(B_{1})}\right)\\
&\leq&
C\left(\|\vec{f}\|_{L^{2}(B_{1})}+\|\vec{f}\|_{K^{1}(B_{1})}+\|g\|_{K(B_{1})}\right),
\end{eqnarray*}
where $C$ is a constant depending on $n$ and $\delta$.
\end{proof}
To continue, we need the following local maximum principle. This theorem has been already proved in \cite{AS1982,CFG1986} where Green function and inverse H\"{o}lder inequality were used. In our paper, we give an another proof by using the approaching method.
\begin{thm}\label{katoconti}
Assume $V\in K_{\eta_{V}}(B_{1})$ with
\begin{equation}\label{vdelta}
\|V\|_{K(B_{1})}\leq\delta
\end{equation}
for some $\delta$ sufficiently small. Then for any weak subsolution $u\in W^{1,2}(B_{1})$ of
\begin{equation}\label{hov}
-\Delta u+Vu=0,\quad\text{in}~~B_{1},
\end{equation}
$u$ is locally bounded with the estimate:
$$\|u\|_{L^{\infty}(B_{\frac{1}{2}})}\leq C\|u\|_{L^{2}(B_{1})},
$$
where $C$ is a constant depending only on $n$ and $\delta$.
\end{thm}

\begin{proof}
We divide our proof into two steps. Firstly, we assume $V\in C^{\infty}(B_{1})$ and satisfies $(\ref{vdelta})$, and we claim that for any $0<q<\infty$,
\begin{equation}\label{3/4}
\|u\|_{L^{\infty}(B_{\frac{1}{2}})}\leq C\|u\|_{L^{q}(B_{\frac{3}{4}})},
\end{equation}
where $C$ is a constant depending only on $n,~q,~\delta$.

In this situation, any solution of $(\ref{hov})$ is locally smooth in $B_{1}$ by classical regularity theory. For any $\varphi\in C_{0}^{\infty}(B_{1})$, we can verify that $u\varphi$ is a solution of
\begin{equation*}
\left\{
\begin{array}{rcll}
-\Delta w+Vw&=&-2\text{div}(u\nabla\varphi)+u\Delta\varphi,\qquad&\text{in}~~ B_{1},\\
w&=&0,\qquad&\text{on}~~\partial B_{1}.
\end{array}
\right.
\end{equation*}
Then by Theorem $\ref{thm5.9}$,
\begin{eqnarray*}
\|u\varphi\|_{L^{\infty}(B_{1})}\leq C\left(\|2u\nabla\varphi\|_{L^{2}(B_{1})}+\|2u\nabla\varphi\|_{K^{1}(B_{1})}+\|u\Delta\varphi\|_{K(B_{1})}\right),
\end{eqnarray*}
where $C$ only depends on $n,~p,~\delta$.
Now for any $\displaystyle\frac{1}{2}\leq t<s\leq \frac{3}{4}$, we take $\varphi\in C_{0}^{\infty}(B_{1})$ with $0\leq\varphi\leq1$ in $B_{1}$, $\varphi=1$ in $B_{t}$, $\varphi=0$ in $B_{1}\backslash B_{s}$, $|\nabla\varphi|\leq\displaystyle\frac{c_{0}}{s-t}$, $|\nabla^{2}\varphi|\leq\displaystyle\frac{c_{0}}{(s-t)^{2}}$,
then we have
\begin{eqnarray*}
\|2u\nabla\varphi\|_{L^{2}(B_{1})}\leq \frac{C}{s-t}\|u\|_{L^{2}(B_{s})},
\end{eqnarray*}
\begin{eqnarray*}
\|2u\nabla\varphi\|_{K^{1}(B_{1})}=\sup\limits_{x\in B_{1}}\int_{B_{1}}\frac{|2u(y)\nabla\varphi(y)|}{|x-y|^{n-1}}dy
\leq\frac{C}{s-t}\|u\|_{L^{q}(B_{s})},
\end{eqnarray*}
\begin{eqnarray*}
\|u\Delta\varphi\|_{K(B_{1})}=\sup\limits_{x\in B_{1}}\int_{B_{1}}\frac{|u(y)\Delta\varphi(y)|}{|x-y|^{n-2}}dy
\leq\frac{C}{(s-t)^{2}}\|u\|_{L^{\frac{q}{2}}(B_{s})},
\end{eqnarray*}
for any $q>n$. Hence we obtain that for any $q>n$,
\begin{equation}\label{st}
\|u\|_{L^{\infty}(B_{t})}\leq \frac{C}{(s-t)^{2}}\|u\|_{L^{q}(B_{s})}.
\end{equation}
Especially, for some fixed $q_{0}>n$, we have
\begin{equation}\label{1/2}
\|u\|_{L^{\infty}(B_{\frac{1}{2}})}\leq C\|u\|_{L^{q_{0}}(B_{\frac{2}{3}})}.
\end{equation}
For any $q>q_{0}$, it is easy to get $(\ref{3/4})$ by H\"{o}lder inequality. In the following we will show that for any $0<q<q_{0}$,
\begin{equation}\label{q0q}
\|u\|_{L^{q_{0}}(B_{\frac{2}{3}})}\leq C\|u\|_{L^{q}(B_{\frac{3}{4}})}.
\end{equation}
We assume $\displaystyle\int_{B_{\frac{3}{4}}}|u|^{q}dx=1$. If $\|u\|_{L^{q_{0}}(B_{\frac{2}{3}})}\leq 1$, then $(\ref{q0q})$ holds naturally. If $\|u\|_{L^{q_{0}}(B_{\frac{2}{3}})}> 1$, we set $I(t)=\left(\displaystyle\int_{B_{t}}|u|^{q_{0}}dx\right)^{\frac{1}{q_{0}}}$ for any $\displaystyle\frac{1}{2}\leq t\leq \frac{3}{4}$, then $I(t)$ is a nondecreasing function. Moreover, from $(\ref{st})$ and the assumption $\displaystyle\int_{B_{\frac{3}{4}}}|u|^{q}dx=1$,  we have for any $\displaystyle\frac{1}{2}\leq t<s\leq \frac{3}{4}$,
\begin{eqnarray*}
I(t)=\left(\displaystyle\int_{B_{t}}|u|^{q_{0}-q+q}dx\right)^{\frac{1}{q_{0}}}
\leq\left(\sup_{B_{t}}|u|\right)^{\theta}\left(\displaystyle\int_{B_{t}}|u|^{q}dx\right)^{\frac{1}{q_{0}}}
\leq\frac{C}{(s-t)^{2\theta}}I(s)^{\theta},
\end{eqnarray*}
where $\theta=\displaystyle\frac{q_{0}-q}{q_{0}}<1$. It follows that
\begin{eqnarray*}
\ln I(t)&\leq& \ln C+2\theta\ln (s-t)^{-1}+\theta\ln I(s)\\
&\leq& C_{1}+\frac{2\theta}{s-t}+\theta\ln I(s).
\end{eqnarray*}
By Lemma $\ref{pee}$ and by using $\|u\|_{L^{q_{0}}(B_{\frac{2}{3}})}> 1$, we obtain that for any $\displaystyle\frac{2}{3}\leq t<s\leq \frac{3}{4}$,
\begin{eqnarray*}
\ln I(t)\leq C\left(\frac{2\theta}{s-t}+C_{1}\right),
\end{eqnarray*}
where $C$ depends only on $q$. It follows that $I(\frac{2}{3})\leq C$ under the assumption $\displaystyle\int_{B_{\frac{3}{4}}}|u|^{q}dx=1$. So in general, we have
\begin{eqnarray*}
\left(\displaystyle\int_{B_{\frac{2}{3}}}|u|^{q_{0}}dx\right)^{\frac{1}{q_{0}}}\leq C\left(\displaystyle\int_{B_{\frac{3}{4}}}|u|^{q}dx\right)^{\frac{1}{q}}.
\end{eqnarray*}
We finish the proof of $(\ref{q0q})$. Combining with $(\ref{1/2})$ and $(\ref{q0q})$, we prove the claim.

Secondly, we consider $V\in K_{\eta_{V}}(B_{1})$. We mollify $\tilde{V}$: the zero extension of $V$, to get $\tilde{V}_{\epsilon}\in C^{\infty}(\mathbb{R}^{n})$. By $(\ref{Veps})$ we have $\|\tilde{V}_{\epsilon}\|_{K(\mathbb{R}^{n})}\leq\delta$ for any $\epsilon>0$. By Lemma 2.1 of \cite{K1994}, there exists a unique solution $u_{\epsilon}\in C_{\text{loc}}^{\infty}(B_{1})$ of
\begin{equation*}
-\Delta v+\tilde{V}_{\epsilon}v=0,\quad\text{in}~~B_{1},\qquad v-u\in W_{0}^{1,2}(B_{1}),
\end{equation*}
such that
$$\|\nabla u_{\epsilon}-\nabla u\|_{L^{2}(B_{1})}+\|u_{\epsilon}-u\|_{L^{2}(B_{1})}\rightarrow 0,\quad \text{as}~~\epsilon\rightarrow0.
$$
It follows that there exists a subsequence of $\{u_{\epsilon}\}$, still denoted by $\{u_{\epsilon}\}$, is convergent to $u$ for almost $x\in B_{1}$.
Moreover, $(\ref{3/4})$ implies that for $\epsilon>0$ small,
\begin{eqnarray*}
\|u_{\epsilon}\|_{L^{\infty}(B_{\frac{1}{2}})}&\leq& C\|u_{\epsilon}\|_{L^{2}(B_{\frac{3}{4}})}\\
&\leq& C\|u_{\epsilon}\|_{L^{2}(B_{1})}\\
&\leq& C\|u_{\epsilon}-u\|_{L^{2}(B_{1})}+C\|u\|_{L^{2}(B_{1})}\\
&\leq& C+C\|u\|_{L^{2}(B_{1})},
\end{eqnarray*}
where $C$ only depends on $n,~\delta$, independently of $\epsilon$. Then we can find a subsequence of $\{u_{\epsilon}\}$, denoted by $\{\tilde{u}_{\epsilon}\}$, is weakly-$\ast$ convergent to $\tilde{u}$, i.e. for any $g\in L^{1}(B_{\frac{1}{2}})$,
$$\int_{B_{\frac{1}{2}}}\tilde{u}_{\epsilon}gdx\rightarrow\int_{B_{\frac{1}{2}}}\tilde{u}gdx,\qquad \text{as}~~\epsilon\rightarrow0.
$$
It follows that $\tilde{u}=u$ in $B_{\frac{1}{2}}$ and
\begin{eqnarray*}
\left|\int_{B_{\frac{1}{2}}}ugdx\right|&=& \lim_{\epsilon\rightarrow0}\left|\int_{B_{\frac{1}{2}}}\tilde{u}_{\epsilon}gdx\right|\\
&\leq&C\|u\|_{L^{2}(B_{1})}\|g\|_{L^{1}(B_{\frac{1}{2}})}\\
\end{eqnarray*}
This means that $u\in L^{\infty}(B_{\frac{1}{2}})$ and
$$\|u\|_{L^{\infty}(B_{\frac{1}{2}})}\leq C\|u\|_{L^{2}(B_{1})}.
$$
We finish the proof.
\end{proof}

Furthermore,  under the assumptions of Theorem $\ref{katoconti}$, if $u\in W^{1,2}(B_{1})$ is a weak solution of $(\ref{hov})$, $u\in L_{\text{loc}}^{\infty}(B_{1})$ can lead to the continuity of $u$. In fact, the local boundedness of $u$ guarantee that $Vu$ still belongs to Kato class. Then by the definition of Kato class and $u$ satisfies $\Delta u=Vu$, we can conclude that $u$ is locally continuous in $B_{1}$. This result has also already been showed in \cite{AS1982,CFG1986}. Next we finish the proof of Theorem $\ref{thm1.2}$.\\

\noindent{\bf Proof of Theorem $\ref{thm1.2}$:}\\

By Theorem $\ref{thm5.9}$, we set $u_{1}\in W_{0}^{1,2}(B_{1})\cap C(\overline{B_{1}})$ be a solution of
\begin{eqnarray*}
\left\{
\begin{array}{rcll}
-\Delta u+Vu&=&-\text{div}\vec{f}+g,\qquad&\text{in}~~ B_{1},\\
u&=&0,\qquad&\text{on}~~\partial B_{1},
\end{array}
\right.
\end{eqnarray*}
Then $u_{1}-u\in W^{1,2}(B_{1})$ satisfies
$$-\Delta(u_{1}-u)+V(u_{1}-u)=0,\quad \text{in}~~B_{1}.
$$
Hence Theorem $\ref{katoconti}$ implies that $u_{1}-u$ is locally continuous in $B_{1}$ and also has the following estimate,
\begin{eqnarray*}
\|u_{1}-u\|_{L^{\infty}(B_{\frac{1}{2}})}&\leq&C\|u_{1}-u\|_{L^{2}(B_{1})}\\
&\leq&C\left(\|u\|_{L^{2}(B_{1})}+\|\vec{f}\|_{L^{2}(B_{1})}+\|\vec{f}\|_{K^{1}(B_{1})}+\|g\|_{K(B_{1})}\right),
\end{eqnarray*}
where the inequality $(\ref{dves})$ is used and $C$ is a constant depending on $n$ and $\|V\|_{K(B_{1})}$. Then $u$ is locally continuous in $B_{1}$ and by triangle inequality, we have
\begin{eqnarray*}
\|u\|_{L^{\infty}(B_{\frac{1}{2}})}&\leq&\|u_{1}-u\|_{L^{\infty}(B_{\frac{1}{2}})}+\|u_{1}\|_{L^{\infty}(B_{\frac{1}{2}})}\\
&\leq&C\left(\|u\|_{L^{2}(B_{1})}+\|\vec{f}\|_{L^{2}(B_{1})}+\|\vec{f}\|_{K^{1}(B_{1})}+\|g\|_{K(B_{1})}\right).
\end{eqnarray*}

\section{Additional observations and remarks}
In the end of this paper, we give some additional observations and remarks to complete our paper. We will give an equivalent condition for Kato class and the proof of Remark $\ref{rem5.15}$. Finally we will give a sufficient condition for our Dini decay condition.
\begin{thm}\label{katodini}
The following two statements are equivalent:\\
(1)$V\in K_{\eta}(B_{1})$ with the modulus of continuity $\eta(r)$;\\
(2)$V$ is $C^{-2,\text{Dini}}$ at point $y$ in $L^{1}$ sense with Dini modulus of continuity $\omega_{y}(r)$ for any $y\in B_{1}$, i.e. there exists $r_{0}>0$, such that for any $0<r\leq r_{0}$ and any $y\in B_{1}$,
$$\int_{0}^{r}\frac{\omega_{y}(s)}{s}ds\leq C(n)\eta(r)\leq C(n)\eta(r_{0})<\infty,
$$
where
$$\omega_{y}(r)=\frac{r^{2}}{|B_{1}\cap B_{r}(y)|}\int_{B_{1}\cap B_{r}(y)}|V(x)|dx,\quad \text{for~any}~~0<r\leq r_{0}.
$$
\end{thm}
\begin{proof}
Without loss of generality, we assume $V=0$ outside $B_{1}$.

$(1)\Rightarrow(2):$ Since $V\in K_{\eta}(B_{1})$, then for some $r_{0}$ fixed,
$$\sup_{y\in B_{1}}\int_{B_{r_{0}}(y)\cap B_{1}}\frac{|V(x)|}{|x-y|^{n-2}}dx\leq \eta(r_{0})<\infty.
$$
For any $0<r\leq r_{0}$,  we calculate $\displaystyle\int_{0}^{r}\frac{\omega_{y}(s)}{s}ds$ straightly to obtain that
\begin{eqnarray*}
\int_{0}^{r}\frac{\omega_{y}(s)}{s}ds&\leq& C\int_{0}^{r}\frac{s}{|B_{s}(y)|}\left(\int_{B_{s}(y)}|V(x)|dx\right)ds\\
&=&C\int_{0}^{r}\frac{1}{s^{n-1}}\left(\int_{0}^{s}\int_{\partial B_{1}}|V(\rho\sigma+y)|\rho^{n-1}d\sigma d\rho\right)ds\\
&=&C\int_{0}^{r}\left(\int_{\partial B_{1}}|V(\rho\sigma+y)|\rho^{n-1}d\sigma\right)\left(\int_{\rho}^{r}\frac{1}{s^{n-1}}ds\right)d\rho\\
&=&C\int_{0}^{r}\int_{\partial B_{1}}\frac{1}{n-2}\left(\frac{1}{\rho^{n-2}}-\frac{1}{r^{n-2}}\right)|V(\rho\sigma+y)|\rho^{n-1}d\sigma d\rho\\
&=&C\left(\int_{B_{r}(y)}\frac{|V(x)|}{|x-y|^{n-2}}dx-\int_{B_{r}(y)}\frac{|V(x)|}{r^{n-2}}dx\right)\\
&\leq&C\int_{B_{r}(y)}\frac{|V(x)|}{|x-y|^{n-2}}dx\\
&\leq&C\eta(r).
\end{eqnarray*}
Hence we prove that $\omega_{y}(r)$ is a Dini modulus of continuity.

$(2)\Rightarrow(1):$ For any $y\in B_{1}$, any $0<r\leq r_{0}$, and Dini modulus of continuity
$$\omega_{y}(r)=\frac{r^{2}}{|B_{1}\cap B_{r}(y)|}\int_{B_{1}\cap B_{r}(y)}|V(x)|dx,
$$
it follows that
\begin{eqnarray*}
\int_{B_{1}\cap(B_{r}(y)\setminus B_{\frac{r}{2}}(y))}\frac{|V(x)|}{|x-y|^{n-2}}dx
&\leq& \frac{2^{n-2}r^{2}}{r^{n}}\int_{B_{1}\cap(B_{r}(y)\setminus B_{\frac{r}{2}}(y))}|V(x)|dx\\
&\leq& \frac{Cr^{2}}{|B_{1}\cap(B_{r}(y)\setminus B_{\frac{r}{2}}(y))|}\int_{B_{1}\cap(B_{r}(y)\setminus B_{\frac{r}{2}}(y))}|V(x)|dx\\
&\leq&\frac{Cr^{2}}{|B_{1}\cap B_{r}(y)|}\int_{B_{1}\cap B_{r}(y)}|V(x)|dx\\
&\leq& C\omega_{y}(r).
\end{eqnarray*}
Similarly,
\begin{eqnarray*}
\int_{B_{1}\cap(B_{\frac{r}{2}}(y)\setminus B_{\frac{r}{4}}(y))}\frac{|V(x)|}{|x-y|^{n-2}}dx\leq C\omega_{y}(\frac{r}{2}),
\end{eqnarray*}
and for any $l=0,1,2,\cdots$,
\begin{eqnarray*}
\int_{B_{1}\cap(B_{\frac{r}{2^{l}}}(y)\setminus B_{\frac{r}{2^{l+1}}}(y))}\frac{|V(x)|}{|x-y|^{n-2}}dx\leq C\omega_{y}(\frac{r}{2^{l}}).
\end{eqnarray*}
Summing the above inequalities over $l$ from 0 to $+\infty$, we have
\begin{eqnarray*}
\int_{B_{1}\cap B_{r}(y)}\frac{|V(x)|}{|x-y|^{n-2}}dx&\leq& C\sum_{l=0}^{\infty}\omega_{y}(\frac{r}{2^{l}})\\
&\leq&4C\int_{0}^{r}\frac{\omega_{y}(s)}{s}ds\leq C(n)\eta(r).
\end{eqnarray*}
Hence we prove that $V\in K_{C\eta}(B_{1})$.
\end{proof}

Next we prove Remark $\ref{rem5.15}$, we firstly recall the following theorem proved by Maz'ya and Verbitsky in \cite{MV02}.
\begin{thm}\label{mv}
(i) Let $d_{\partial \Omega}(x)=\operatorname{dist}(x, \partial \Omega)$, and let
$$
V=\operatorname{div} \vec{\Gamma}+d_{\partial \Omega}^{-1} \Gamma_0
$$
where $\vec{\Gamma}=(\Gamma_1, \ldots, \Gamma_n)$ and $\Gamma_i \in M(W^{1,2}_{0}(\Omega) \rightarrow L^2(\Omega))$ for $i=0,1, \ldots, n$. Suppose that the following Hardy inequality holds for any $u\in C_{0}^{\infty}(\Omega)$:
$$\int_{\Omega}\frac{|u(x)|^{2}}{d_{\partial\Omega}(x)^{2}}dx\leq C\int_{\Omega}|\nabla u|^{2}dx.
$$
Then $V \in M(W^{1,2}_{0}(\Omega) \rightarrow W^{-1,2}(\Omega))$ and
$$
\|V\|_{M(W^{1,2}_{0}(\Omega) \rightarrow W^{-1,2}(\Omega))} \leqslant C \sum_{0 \leq i \leq n}\|\Gamma_i\|_{M(W^{1,2}_{0}(\Omega) \rightarrow L^2(\Omega))} .
$$
(ii) Conversely, if $V \in M(W^{1,2}_{0}(\Omega) \rightarrow W^{-1,2}(\Omega))$, then there exist $\vec{\Gamma}=\left(\Gamma_1, \ldots, \Gamma_n\right)$ and $\Gamma_0$ such that $\Gamma_i \in M(W^{1,2}_{0}(\Omega) \rightarrow L^2(\Omega))$ for $i=0,1, \ldots, n$, and $V=\operatorname{div} \vec{\Gamma}+d_{\partial \Omega}^{-1} \Gamma_0$. Moreover,
$$
\sum_{0 \leq i \leq n}\|\Gamma_i\|_{M(W^{1,2}_{0}(\Omega) \rightarrow L^2(\Omega))} \leq C\|V\|_{M(W^{1,2}_{0}(\Omega) \rightarrow W^{-1,2}(\Omega))}.
$$
\end{thm}

\noindent{\bf Proof of Remark $\ref{rem5.15}$:}\\

$(1)\Rightarrow(2)$: In fact, we take $\psi\in W_{0}^{1,2}(B_{r}(y))$ in $(\ref{eqv1})$, it follows that
$$
|\langle V\psi,\varphi\rangle|\leq C\omega(r)\|\nabla\psi\|_{L^{2}(B_{r}(y))}\|\nabla\varphi\|_{L^{2}(B_{r}(y))}.
$$
This implies that $V\in M(W^{1,2}_{0}(B_{r}(y)) \rightarrow W^{-1,2}(B_{r}(y)))$. Then by Theorem $\ref{mv}$ (ii), there exists $\vec{\Gamma}_{r,y},\Gamma_{r,y,0}$ such that $\Gamma_{r,y,i}\in M(W^{1,2}_{0}(B_{r}(y)) \rightarrow L^2(B_{r}(y)))$ for $0\leq i\leq n$, and
$$V=\text{div}\vec{\Gamma}_{r,y}+d^{-1}_{\partial B_{r}(y)}(x)\Gamma_{r,y,0}$$ in $B_{r}(y)$. Furthermore,
$$\sum_{0 \leq i \leq n}\|\Gamma_{r,y,i}\|_{M(W^{1,2}_{0}(B_{r}(y)) \rightarrow L^2(B_{r}(y)))} \leq C\omega(r).
$$

$(2)\Rightarrow(1)$: We assume that
$$\int_{B_{s}(y)}|\Gamma_{s,y,i}|^{2}\varphi^{2}dx\leq C\omega(s)^{2}\|\nabla\varphi\|_{L^{2}(B_{s}(y))}^{2},
$$ for any $\varphi\in W_{0}^{1,2}(B_{s}(y))$ and any $r<s\leq2r$.
Now we take cut-off function $\xi\in C^{\infty}_{0}(B_{s}(y))$ with $\xi\equiv1$ in $B_{r}(y)$, $0\leq\xi\leq1$ and $|\nabla\xi|\leq\displaystyle\frac{C}{s-r}$. Then for any $\psi\in W^{1,2}(B_{s}(y))$, if we take $\varphi=\xi\psi\in W_{0}^{1,2}(B_{s}(y))$ in the above inequality, it follows that
\begin{eqnarray*}
\int_{B_{r}(y)}|\Gamma_{s,y,i}|^{2}\psi^{2}dx&=&\int_{B_{r}(y)}|\Gamma_{s,y,i}|^{2}\xi^{2}\psi^{2}dx\\
&\leq&\int_{B_{s}(y)}|\Gamma_{s,y,i}|^{2}(\xi\psi)^{2}dx\\
&\leq& C\omega(s)^{2}\|\nabla(\xi\psi)\|_{L^{2}(B_{s}(y))}^{2}\\
&\leq& C\omega(s)^{2}\left(\frac{\|\psi\|_{L^{2}(B_{s}(y))}^{2}}{(s-r)^{2}}+\|\nabla\psi\|_{L^{2}(B_{s}(y))}^{2}\right)\\
&\leq& C\omega(s)^{2}\|\psi(\cdot+y)\|_{L_{r,s}^{1,2}}^{2}.
\end{eqnarray*}
Hence for any $0<r\leq\displaystyle\frac{1}{2}$, for any $y\in B_{1}$ satisfying $B_{2r}(y)\subset B_{1}$, $\psi\in W^{1,2}(B_{1})$,  and $\varphi\in W_{0}^{1,2}(B_{1})$ with $\text{supp}\{\varphi\}\subset \overline{B_{r}(y)}$, we have  by H\"{o}lder inequality and Hardy inequality,
\begin{eqnarray*}
|\langle V\psi,\varphi\rangle|&\leq& \left|\int_{B_{r}(y)}\vec{\Gamma}_{s,y}\cdot\nabla(\psi\varphi)dx\right|+\left|\int_{B_{r}(y)}d^{-1}_{\partial B_{s}(y)}(x)\Gamma_{s,y,0}\psi\varphi dx\right|\\
&\leq&\left(\int_{B_{r}(y)}|\vec{\Gamma}_{s,y}|^{2}|\psi|^{2}dx\right)^{\frac{1}{2}}\left(\int_{B_{r}(y)}|\nabla\varphi|^{2}\right)^{\frac{1}{2}}
+\left(\int_{B_{r}(y)}|\vec{\Gamma}_{s,y}|^{2}|\varphi|^{2}dx\right)^{\frac{1}{2}}\left(\int_{B_{r}(y)}|\nabla\psi|^{2}\right)^{\frac{1}{2}}\\
&&+\left(\int_{B_{r}(y)}|\Gamma_{s,y,0}|^{2}|\psi|^{2}dx\right)^{\frac{1}{2}}\left(\int_{B_{r}(y)}\frac{\varphi^{2}}{d^{2}_{\partial B_{s}(y)}(x)}dx\right)^{\frac{1}{2}}\\
&\leq&C\omega(s)\|\psi(\cdot+y)\|_{L_{r,s}^{1,2}}\|\nabla\varphi\|_{L^{2}(B_{r}(y))}.
\end{eqnarray*}

We finish the proof.\\

In this paper, it shows that Dini decay condition $(\ref{eqv1})$ leads to the continuity of the solution. Maybe when $V\in M(W^{1,2}(B_{1}),W^{-1,2}(B_{1}))$, $(\ref{eqv1})$ is not easily to verified. while Remark $\ref{rem5.15}$ gives an equivalent condition  $(\ref{eqv2})$ which can be more conveniently  checked. Next, we give a sufficient condition of $(\ref{eqv1})$.
\begin{thm}
Assume $V=\emph{div}\vec{\Gamma}$ in $B_{1}$
where $\vec{\Gamma}=(\Gamma_1, \ldots, \Gamma_n)$. There exists $0<r_{0}\ll1$, for any $0 < r\leq r_{0}$, $\Gamma_{i}(1\leq i\leq n)$ satisfies
\begin{equation}\label{eqv3}
\sup_{y\in B_{1}}\left(\frac{1}{|B_{1}\cap B_{r}(y)|}\int_{B_{1}\cap B_{r}(y)}|\Gamma_{i}|^{2}dx\right)^{\frac{1}{2}}\leq C\frac{\omega_{1}(r)}{r},
\end{equation}
where $\omega_{1}(r)$ is a modulus of continuity satisfying
\begin{equation}\label{omega1r}
\displaystyle\int_{0}^{r_{0}}\frac{\omega_{1}(r)^{2}}{r}dr<\infty, \quad \displaystyle\int_{0}^{r_{0}}\frac{1}{r}\sqrt{\displaystyle\int_{0}^{r}\frac{\omega_{1}(s)^{2}}{s}ds}dr<\infty.
\end{equation}
Then $V$ satisfies $(\ref{eqv1})$ with $\omega(r)=\sqrt{\displaystyle\int_{0}^{r}\displaystyle\frac{\omega_{1}(s)^{2}}{s}ds}$.
\end{thm}
\begin{proof}
By Theorem $\ref{katodini}$ and $(\ref{omega1r})$, we can know that $|\Gamma_{i}|^{2}(1\leq i\leq n)$ belongs to $K_{\tilde{\omega}_{1}}(B_{1})$ where $\tilde{\omega}_{1}(r)=C\displaystyle\int_{0}^{r}\frac{\omega_{1}(s)^{2}}{s}ds$ and $\sqrt{\tilde{\omega}_{1}(r)}$ is a Dini modulus of continuity. By Lemma $\ref{lm5.4}$, it follows that for any $u\in W_{\text{loc}}^{1,2}(B_{1})$, for any $0<r<s\leq 2r\leq \displaystyle r_{0}$ and $x_{0}\in B_{1}$ with $\overline{B_{2r}(x_{0})}\subseteq B_{1}$,
\begin{eqnarray*}
\int_{B_{r}(x_{0})}|\Gamma_{i}(x)|^{2}u^{2}(x)dx&\leq& C\left(\sup_{x\in B_{s}(x_{0})}\int_{B_{s}(x)}\frac{|\Gamma_{i}(y)|^{2}}{|x-y|^{n-2}}dy\right)
\left(\frac{4}{(s-r)^{2}}\|u\|_{L^{2}(B_{s}(x_{0}))}^{2}+\|Du\|_{L^{2}(B_{s}(x_{0}))}^{2}\right)\\
&\leq&C\tilde{\omega}_{1}(s)\left(\frac{1}{(s-r)^{2}}\|u\|_{L^{2}(B_{s}(x_{0}))}^{2}+\|Du\|_{L^{2}(B_{s}(x_{0}))}^{2}\right)\\
&\leq&C\tilde{\omega}_{1}(s)\|u(\cdot+x_{0})\|^{2}_{L_{r,s}^{1,2}}.
\end{eqnarray*}
Then for any $\psi\in W^{1,2}(B_{1})$, $\varphi\in W_{0}^{1,2}(B_{1})$ with $\text{supp}\{\varphi\}\subset \overline{B_{r}(x_{0})}$,
\begin{eqnarray*}
|\langle V\psi,\varphi\rangle|&=&\left|\int_{B_{r}(x_{0})}\vec{\Gamma}\cdot\nabla(\psi\varphi)dx\right|\\
&\leq&C\sqrt{\tilde{\omega}_{1}(s)}\|\psi(\cdot+x_{0})\|_{L_{r,s}^{1,2}}\|\nabla\varphi\|_{L^{2}(B_{r}(x_{0}))},
\end{eqnarray*}
for any $r<s\leq2r$. Combining with $(\ref{omega1r})$, this means that $(\ref{eqv1})$ holds for $V$ with $\omega(r)=\sqrt{\displaystyle\int_{0}^{r}\displaystyle\frac{\omega_{1}(s)^{2}}{s}ds}$.

\end{proof}
\noindent{\bf Data availability} No data was used for the research described in the article.

\end{document}